\documentclass[11pt,a4paper]{amsart}
\usepackage{geometry}
\usepackage[colorlinks, linkcolor=red, anchorcolor=blue, citecolor=blue]{hyperref}
\usepackage{color}     
\usepackage{makecell}         
\usepackage{amsthm,bm} 
\usepackage{enumerate}
\usepackage{amsfonts}
\usepackage{amsmath}
\usepackage{mathrsfs}
\usepackage{amssymb}
\usepackage{amsbsy}
\usepackage{ifthen}
\usepackage[all]{xy}
\usepackage{extarrows}
\usepackage{indentfirst}        
\usepackage{latexsym}        

\usepackage{graphicx}
\usepackage{cases}
\usepackage{pifont}
\usepackage{txfonts}
\usepackage{xcolor}
\usepackage{multirow}
\usepackage{caption, subcaption}
\usepackage{tikz}

\newcounter{maint}

\usepackage[vcentermath]{youngtab}

\numberwithin{equation}{section}

\def\k{\Bbbk}
\def\id{\mathrm{id}}

\allowdisplaybreaks[4] 

\begin{document}

\newtheorem{theorem}{Theorem}[section]

\newtheorem{lemma}[theorem]{Lemma}

\newtheorem{corollary}[theorem]{Corollary}
\newtheorem{proposition}[theorem]{Proposition}

\theoremstyle{remark}
\newtheorem{remark}[theorem]{Remark}

\theoremstyle{definition}
\newtheorem{definition}[theorem]{Definition}

\theoremstyle{definition}
\newtheorem{conjecture}[theorem]{Conjecture}

\newtheorem{example}[theorem]{Example}
\newtheorem{problem}[theorem]{Problem}
\newtheorem{question}[theorem]{Question}

\title[on Nichols algebras associated to Near-rack solutions]
{on Nichols algebras associated to Near-rack solutions of the Yang-Baxter equation}

\author[Shi]{Yuxing Shi }
\address{School of Mathematics and Statistics, and  Jiangxi Provincial Center for Applied Mathematics,  Jiangxi Normal University, Nanchang, Jiangxi 330022, People's Republic of China}
\email{yxshi@jxnu.edu.cn}

\makeatletter
\@namedef{subjclassname@2020}{\textup{2020} Mathematics Subject Classification}
\makeatother

\begin{abstract}
Let $(X, r)$ be  any set-theoretical non-degenerate solution  of the Yang-Baxter equation
and $(X, \tilde r)$ be the derived solution of $(X, r)$. 
As for any braided vector space $(W_{X, r}, c)$ associated to  $(X, r)$, 
is it possible to find some  braided vector space $(W_{X, \tilde r}, \tilde c)$ 
which is   t-equivalent to $(W_{X, r}, c)$? In case that  $(X, r)$ is a near-rack solution, we 
give a sufficient condition to make an affirmative answer to the question. 
Examples of t-equivalence are constructed, hence finite dimensional Nichols algebras 
are obtained. In particular, all finite dimensional Nichols algebras associated to 
involutive near-rack solutions are classified. 
\end{abstract}
\maketitle

\section{Introduction}

The paper is a contribution to the classification program of finite dimensional Nichols algebras. 
The Nichols algebra was first discovered by Nichols \cite{MR506406} and it has 
 broad applications in mathematics and 
physics, see for example  \cite{woronowicz1989differential, MR1227098, MR1632802, andruskiewitsch2001pointed, Lentner2021, Kapranov2020, Meir2022}. 
The Nichols algebra $\mathfrak{B}(V)$ is a connected graded 
braided Hopf algebra which is generated as an algebra by a  (rigid) braided vector space $V$ and completely determined by its braiding. The Yetter-Drinfeld modules over a Hopf algebra $H$ are examples 
of rigid braided vector spaces. 
As for Yetter-Drinfeld modules over a finite group algebra, their braidings  can be described by racks \cite{Andruskiewitsch2003MR1994219}. 
If any Yetter-Drinfeld module over a Hopf algebra $H$ is of set-theoretical type, then we say 
$H$ is set-theoretical. While studying Nichols algebras over the Suzuki algebras 
\cite{Shi2019, Shi2020even, Shi2020odd}, the author has come to realize that 
the Suzuki algebras are set-theoretical and finite dimensional 
semisimple Hopf algebras arising from abelian extension 
are conjectured to be  set-theoretical \cite{shi2023_NearRack}. 
Some problems were posed in \cite{shi2023_NearRack}: classify set-theoretical Hopf algebras and 
finite dimensional Nichols algebras of set-theoretical type. 

A braided vector space is said to be of set-theoretical type if it is associated to a 
set-theoretical solution of the Yang-Baxter equation. The Yang-Baxter equation was first introduced in physics \cite{Yang1967, Baxter1972} and set-theoretical solutions are closely related to  many algebraic structures such as  Hopf algebras, Nichols algebras, racks, (skew) braces, (relative) Rota-Baxter groups and so on; see for example \cite{Andruskiewitsch2003MR1994219, Rump2007, zbMATH06713497, 
Guo2021, Caranti2023, Bai2023, zbMATH07473025}.

It is obvious that set-theoretical type includes rack type. 
In case the braiding $c$ is of diagonal type (also called abelian rack type), the classification was completed  by Heckenberger \cite{heckenberger2009classification}
based on the theory of reflections  \cite{Heckenberger[2020]copyright2020}  and  Weyl groupoid \cite{MR2207786}, and the minimal presentation of these Nichols algebras was obtained by Angiono
\cite{Angiono2013, MR3420518}. In case $V$ is a  non-simple 
semisimple Yetter-Drinfeld module over non-abelian group
algebras, the classification were almost 
finished by  Heckenberger and  Vendramin \cite{Heckenberger2017, MR3605018}
under some technical assumptions.  
As for the case of indecomposable braided vector spaces of rack type, the classification problem is wildly open, except few examples are known, please refer to 
\cite[Table 9.1]{Heckenberger2015} and the references therein. 
Heckenberger, Lochmann and Vendramin conjectured  
that any finite dimensional elementary Nichols algebra of group type 
is $bg$-equivalent to one of those listed in 
\cite[Table 9.1]{Heckenberger2015}.

Let $(X, r)$  be a non-degenerate set-theoretical solution of the Yang-Baxter equation, then 
we can construct  braided vector spaces $(W_{X, r}, c)$ of set-theoretical type
and $(W_{X, \tilde r}, \tilde c)$ of rack type associated to 
the derived solution $(X, \tilde r)$ of $(X, r)$ with $\tilde r=TrT^{-1}$, 
where $W_{X, r}=W_{X, \tilde r}$ as vector spaces and  $T(x, y)=(\tau_y(x), y)$ for $x, y\in X$. 
The following question was  first considered by 
Andruskiewitsch and Gra\~{n}a in 
\cite[Example 5.11]{Andruskiewitsch2003MR1994219}. 
\begin{question}\label{MainQuestion}
As for any braided vector space $(W_{X, r}, c)$, 
is it possible to find some  braided vector space $(W_{X, \tilde r}, \tilde c)$ 
which is   t-equivalent to $(W_{X, r}, c)$? 
\end{question}
If so, then Nichols algebras $\mathfrak{B}(W_{X, r}, c)$ and $\mathfrak{B}(W_{X, \tilde r}, \tilde c)$
are isomorphic as graded vector spaces. So we can 
make use of  classification results of finite-dimensional Nichols algebras of rack type. 
Furthermore, generators and relations of $\mathfrak{B}(W_{X, r}, c)$ can be obtained according to those 
of $\mathfrak{B}(W_{X, \tilde r}, \tilde c)$ via t-equivalence, see Remark \ref{RelationNicholsAlg}. 
Unfortunately, this question is extremely complicated in general. 
In \cite{shi2023_NearRack}, the author showed that two braided vector spaces $(W, c)$ and $(W, \tilde c)$
are t-equivalent if there are two invertible maps  $\varphi_i: W \to W$ with $i\in\{1, 2\}$ such that 
$\tilde c=(\varphi_1^{-1}\otimes \varphi_2^{-1})c(\varphi_1\otimes \varphi_2)
=(\varphi_2^{-1}\otimes \varphi_1^{-1})c(\varphi_2\otimes \varphi_1)$. 
To make the braiding $\tilde c$ possibly satisfied the property, we require that the solution
$(X, r)$ is given by $r(x, y)=(\sigma_x(y), \tau(x))$ for any $x, y\in X$, where
$\tau\neq {\rm id}=\tau^2$. We call this solution a near-rack solution. 
As for near-rack solutions, a sufficient condition is given to make  
$\mathfrak{B}(W_{X, r}, c)$ t-equivalent to  some $\mathfrak{B}(W_{X, \tilde r}, \tilde c)$, see Theorem \ref{KeyTheorem}. 
If $(X, r)$ is the one in Examples  \ref{123Alt4}, \ref{1234S4}, \ref{Aff52}, \ref{Aff53}, 
\ref{12S4},
\ref{Aff73}, \ref{Aff75},  then the answer for Question \ref{MainQuestion} is affirmative. 
If  $(X, r)$ is the one in Examples \ref{T_equiv_DihedralD3}, \ref{T_Equiva_Second1234S4}, 
\ref{T_Equiva_First12S4}, 
then some constraints need to be added  on $W_{X, r}$ to give an affirmative answer to Question \ref{MainQuestion}.  
Those near-racks are related with racks: the dihedral rack $\Bbb D_3$, affine racks ${\rm Aff}(5,2)$, 
${\rm Aff}(5,3)$, ${\rm Aff}(5,3)$, ${\rm Aff}(7,5)$, and conjugate classes $(1, 2)^{\Bbb S_4}$, 
$(1, 2,3,4)^{\Bbb S_4}$, $(1,2,3)^{{\rm Alt}_4}$, which are in the list of finite dimensional Nichols algebras 
of non-abelian rack type \cite[Table 9.1]{Heckenberger2015}. 
Hence finite dimensional Nichols algebras are obtained via those t-equivalences. 

If $(X, r)$ is an involutive near-rack solution, then $\sigma_x=\tau$ for any $x\in X$, 
$\tau^2={\rm id}\neq \tau$ according to Proposition \ref{All_Involutive_NearRacks}.
The answer to Question \ref{MainQuestion} is affirmative for any involutive near-rack solution. 
Furthermore, all finite dimensional Nichols algebras associated to involutive near-rack solutions 
are classified, see Propositions \ref{LengthTwo}--\ref{Length_k}. 

The paper is organized as follows. 
In Section 2, we recall some basic notations and results on the Nichols algebra, 
set-theoretical solution of the Yang-Baxter equation and Rack. 
In Section 3,  we introduce the near-rack solution of the Yang-Baxter equation. 
Examples of near-rack solutions are presented. In particular, 
if $(X, r)$ is an involutive near-rack solution, then $\sigma_x=\tau$ for any $x\in X$, 
$\tau^2={\rm id}\neq \tau$. In Section 4,  we  give a sufficient condition to 
make an affirmative answer to Question \ref{MainQuestion} for the case that  $(X, r)$ is a near-rack solution. 
Examples of t-equivalence are constructed in Appendix. 
In Section 5, all finite dimensional Nichols algebras associated to 
involutive near-rack solutions are classified. 

\section{Preliminaries}
\subsection{Notations} 
Let $\k$ be  an algebraicaly  closed field of characteristic $0$ and $\Bbbk^\times$ be
$\Bbbk-\{0\}$. For integers $a\leq b$, we let 
$[a, b]\coloneqq\{a, a+1,\cdots,b-1, b\}$.  Let $\Bbb S_n$ denote the symmetric group consisting of 
permutations of $[1,n]$. We denote a cycle in a symmetric group as 
$\theta=(a_1, \cdots, a_t)$ where $\theta$ sends $a_1$ to $a_2$, and so on. For 
integer $m\geq 2$, $\Bbb G_m^\prime$ is the set of $m$-th primitive roots of unity. 

\subsection{Nichols algebras}
\begin{definition}
Let $V$ be a vector space. Then $(V, c)$ is called a braided vector space, if the linear isomorphism 
$c: V\otimes V\to V\otimes V$ is a solution of the braid equation, that is
$
(c\otimes {\rm id})({\rm id}\otimes c)(c\otimes {\rm id})=
({\rm id}\otimes c)(c\otimes {\rm id})({\rm id}\otimes c).
$
\end{definition}

\begin{definition}
Let $(V, c)$ be a braided vector space and $T(V)$ be the tensor algebra. 
Set $s_i=(i, i+1)\in \Bbb S_n$, 
$c_i={\rm id}^{\otimes (i-1)}\otimes c\otimes {\rm id}^{\otimes (n-i-1)}\in{\rm End}_{\Bbbk}(V^{\otimes n})$
for $i\in[1, n-1]$. The quantum symmetrizer  
\begin{equation}
\mathfrak S_{n,c}=\sum_{\sigma\in\Bbb S_n}\mathcal{T}(\sigma)\in{\rm End}_{\Bbbk}(V^{\otimes n}), 
\end{equation}
where $\mathcal T: \Bbb S_n\to {\rm End}_{\Bbbk}(V^{\otimes n})$ is defined as 
$\mathcal T(s_{i_1}s_{i_2}\cdots s_{i_t})=c_{i_1}c_{i_2}\cdots c_{i_t}$ for any reduced expression 
$\sigma=s_{i_1}s_{i_2}\cdots s_{i_t}\in\Bbb S_n$. Then the Nichols algebra
\begin{equation}
\mathfrak{B}(V, c)=\Bbbk\oplus V\oplus \bigoplus_{n=2}^{\infty}T^n(V)/\ker \mathfrak S_{n,c}. 
\end{equation}
\end{definition}
\begin{remark}
The notation $\mathfrak B(V)$ is short for $\mathfrak B(V, c)$. In general,  
the braided vector space $(V, c)$ is required to be rigid, which means that $(V, c)$ can be realized as a Yetter-Drinfeld module 
over some Hopf algebra. There are several equivalent definitions of the Nichols algebra $\mathfrak B(V, c)$
for different purposes. For more informations, please refer to  
\cite{Heckenberger[2020]copyright2020, andruskiewitsch2001pointed, MR1396857, MR1632802}. 
\end{remark}
\begin{example}
$\mathfrak S_{3, c}={\rm id}+c_1+c_2+c_1c_2+c_2c_1+c_1c_2c_1$. 
\end{example}

\begin{definition}\cite{Heckenberger[2020]copyright2020}
We say a braided vector space $(V, c)$ is of diagonal type, if $V=\bigoplus_{i\in I} \Bbbk v_i$ with 
$c(v_i\otimes v_j)=q_{ij}v_j\otimes v_i$ for any $i, j\in I$ and $q_{ij}\in\Bbbk^\times$. 
A generalized Dynkin  diagram of $(V, c)$ is a graph with $|I|$ vertices, where the $i$-th
vertex is labeled with $q_{ii}$ for all $i\in I$; further,  if $q_{ij}q_{ji}\neq  1$, then there is an edge between the $i$-th and $j$-th vertex labeled with $q_{ij}q_{ji}$. 
\end{definition}

\begin{definition}\cite[Definition 5.10]{Andruskiewitsch2003MR1994219}
\label{Remark_T_Equivalent}
Two braided vector spaces $(V, c)$ and $(W, \tilde c)$ are t-equivalent if there is a collection of 
linear isomorphisms $U_n: V^{\otimes n}\to W^{\otimes n}$ intertwining the corresponding representations 
of the braid group $\Bbb B_n$ for all $n\geq 2$. 
\end{definition}
\begin{remark}\cite[Lemma 6.1]{Andruskiewitsch2003MR1994219}
If braided vector space $(V, c)$ is t-equivalent to  $(W, \tilde c)$, then 
the corresponding Nichols algebras  $\mathfrak{B}(V)$ and $\mathfrak{B}(W)$ are isomorphic as 
graded vector spaces. Hence $\dim \mathfrak{B}(V)=\dim \mathfrak{B}(W)$.
\end{remark}

\begin{remark}\label{RelationNicholsAlg}
Definition \ref{Remark_T_Equivalent} tells us that $U_n\tilde c_i=c_iU_n$ which implies that 
$\mathfrak S_{n, c}=U_n\mathfrak S_{n, \tilde c}U_n^{-1}$. In other words, 
$v_n\in V^{\otimes n}$ is a relation of $\mathfrak{B}(V)$ if and only if $U_n^{-1}(v_n)$ is a relation of 
$\mathfrak{B}(W)$. 
\end{remark}

\begin{lemma}\label{MainLemma}\cite{shi2023_NearRack}
Let $(V, c)$ be a braided vector space, and $\varphi_1, \varphi_2: V\to V$ be two invertible linear maps. 
If 
\[
\tilde c=(\varphi_1^{-1}\otimes \varphi_2^{-1})c(\varphi_1\otimes \varphi_2)
=(\varphi_2^{-1}\otimes \varphi_1^{-1})c(\varphi_2\otimes \varphi_1),
\]
 then  $(V, \tilde c)$ is a braided vector space and is t-equivalent to $(V, c)$.
\end{lemma}

\subsection{Set-theoretical solution of the Yang-Baxter equation}
A set-theoretical solution of the Yang-Baxter equation is a  pair $(X, r)$, where 
$X$   is a non-empty set and $r: X\times X\rightarrow X\times X$ is a bijective map
such that
\[
(r\times {\rm id})({\rm id}\times r)(r\times {\rm id})
=({\rm id}\times r)(r\times {\rm id})({\rm id}\times r)
\]
holds. By convention, we write 
\[
r(i, j)=(\sigma_i(j), \tau_j(i)),\quad  \forall i, j\in X.
\]
Two solutions $(X, r)$ and $(X, r^\prime)$ are isomorphic if there exists an invertible map 
$\varphi: X\to X$ such that $(\varphi\times \varphi)r=r^\prime (\varphi\times \varphi)$.
A solution $(X, r)$ is non-degenerate if all the maps $\sigma_i: X\rightarrow X$
and $\tau_i: X\rightarrow X$ are bijective for all $i\in X$, and involutive if 
$r^2=\id_{X\times X}$.

\begin{definition}\label{BraidedVectorSpace}
\cite[Lemma 5.7]{Andruskiewitsch2003MR1994219}
Let $(X, r)$ be a non-degenerate set-theoretical  solution of the Yang-Baxter equation, $|X|=m\in\Bbb Z^{\geq 2}$.
Then $W_{X, r}=\bigoplus_{i\in X}\k w_i$ is a braided vector 
space with the braiding given by 
\begin{align}\label{YBEquation}
c(w_{i}\otimes w_j)
&=R_{i,j} w_{\sigma_i(j)}\otimes w_{\tau_j(i)}, \quad 
\text{where}\,\, R_{i,j}\in\k^\times\,\,\text{and} \\ 
R_{i, j}R_{\tau_j(i), k}R_{\sigma_i(j),\sigma_{\tau_j(i)}(k)}
&=R_{j,k}R_{i,\sigma_j(k)}R_{\tau_{\sigma_j(k)}(i),\tau_k(j)}, 
\quad \forall i, j, k\in X. \label{YBEquation}
\end{align}
The braided vector space  $W_{X, r}$ is called of set-theoretical type. 
\end{definition}

\begin{remark}
The braiding of $W_{X,r}$ is rigid according to 
\cite[Lemma 3.1.3]{Schauenburg1992}.
\end{remark}

\subsection{Racks}
\begin{definition}
Let $X$ be a non-empty set, then $(X,\rhd)$ is a rack if $\rhd: X\times X\to X$ is a function, such that 
$\phi_i: X\to X$, $\phi_i(j)=i\rhd j$,  is bijection for all $i\in X$ and 
\begin{align}
i\rhd(j\rhd k)=(i\rhd j)\rhd (i\rhd k), \quad \forall i, j, k\in X.
\end{align}
\end{definition}

\begin{remark}\cite{Brieskorn1988}
$(X, \rhd)$ is a rack if and only if $(X, r)$ is a set-theoretical solution of the Yang-Baxter equation, where 
$r(x, y)=(x\rhd y, x)$ for all $x, y\in X$. 
\end{remark}

\begin{example}\cite{zbMATH01585085, LuJianghua2000MR1769723}
\label{SetRackRelations}
Let $(X, r)$ be a non-degenerate set-theoretical solution of the Yang-Baxter equation. Then 
$(X, \rhd)$ is a rack with $x\rhd y=\tau_x\sigma_{\tau_y^{-1}(x)}(y)$ for all $x, y\in X$. 
Let $T: X\times X\rightarrow X\times X$ with $T(x, y)=(\tau_y(x), y)$, then 
$T$ is invertible and $T^{-1}(x, y)=(\tau_y^{-1}(x), y)$. We have 
\begin{align}
TrT^{-1}(x, y)=(x\rhd y, x).
\end{align}
The solution $(X, TrT^{-1})$ is called the derived solution of $(X, r)$. 
\end{example}

\begin{remark}\cite[Example 5.11]{Andruskiewitsch2003MR1994219}
Let $W_{X, r}$ be a braided vector space of set-theoretical type and $(X, \tilde r)$ be 
the derived solution of $(X, r)$. Set $\tilde c(w_i\otimes w_j)=q_{ij}w_{i\rhd j}\otimes w_i$, where
$q_{ij}=R_{\sigma_j^{-1}(i), j}$ for all $i, j\in X$. 
If $q_{\sigma_k(i), \sigma_k(j)}=q_{ij}$ for all $i, j, k\in X$, then 
the braided vector spaces $(W_{X, r}, c)$ and $(W_{X, \tilde r}, \tilde c)$ are 
t-equivalent.  
\end{remark}

\begin{example}\label{TypeD}
Let $\Bbb D_n=(\Bbb Z_n, \rhd)$, $i\rhd j=2i-j\in\Bbb Z_n$ for all $i, j\in \Bbb Z_n$. 
Then $\Bbb D_n$ is a rack and called a dihedral rack.  
\end{example}

\begin{example}\cite{Andruskiewitsch2003MR1994219}
Let $A$ be an abelian group and $f\in {\rm Aut}\, A$. The affine rack ${\rm Aff}(A, f)$ is the set 
$A$ with the operation
\[
a\rhd b=f(b)+({\rm id}-f)(a),\qquad \forall a, b\in A. 
\]
\end{example}

\section{The near-rack solution of the Yang-Baxter equation}
\begin{definition}
Let $(X, r)$ be a non-degenerate set-theoretical solution of the Yang-Baxter equation. 
Denote $r(x, y)=(\sigma_x(y), \tau_y(x))$. If $\tau=\tau_y$ for all $y\in X$, $\tau\neq {\rm id}$ 
and $\tau^2={\rm id}$, then we call $(X, r)$ a near-rack solution of the Yang-Baxter equation. 
\end{definition}

\begin{remark}
According to the definition of set-theoretical solution of the Yang-Baxter equation, we have
\begin{equation}\label{YB_id}
\sigma_{\sigma_x(y)}\sigma_{\tau(x)}=\sigma_x\sigma_y,\qquad
\tau\sigma_x=\sigma_{\tau(x)}\tau, \qquad \forall x, y\in X.
\end{equation}
\end{remark}

\begin{lemma}
Let $(X, r)$ be a near-rack solution of the Yang-Baxter equation and 
$\tilde r=(\tau\times {\rm id})r(\tau\times {\rm id})$. Then $(X, \tilde r)$ is the derived solution of $(X, r)$ and 
\[
\tilde r=(\tau\times {\rm id})r(\tau\times {\rm id})=({\rm id}\times \tau)r({\rm id}\times \tau).
\] 
\end{lemma}

\begin{remark}
This Lemma shows that it is possible to establish a t-equivalent relationship between braided vector spaces $W_{X, r}$
and $W_{X, \tilde r}$. 
\end{remark}

\begin{proof}
For any $x, y\in X$, we have 
\begin{align*}
\tilde r(x, y)&=(\tau\times {\rm id})r(\tau\times {\rm id})(x, y)
=(\tau\times {\rm id})r(\tau(x), y)\\
&=(\tau\times {\rm id})(\sigma_{\tau(x)}(y), \tau^2(x))\\
&=(\tau\sigma_{\tau(x)}(y), x)\\
&=({\rm id}\times \tau)(\tau^2\sigma_x\tau(y), \tau(x))=({\rm id}\times \tau)(\sigma_x\tau(y), \tau(x))\\
&=({\rm id}\times \tau)r(x, \tau(y))\\
&=({\rm id}\times \tau)r({\rm id}\times \tau)(x, y).
\end{align*}
\end{proof}

\begin{example}\cite{Gu1997}
Let $G$ be a group, $\tau$ be a map from $G$ to itself. Then 
\[
r(x, y)=\left(xy\tau(x)^{-1}, \tau(x)\right), \quad \forall x, y\in G, 
\]
is a solution of the Yang-Baxter equation if and only if 
\[\tau\left(xy\tau(x)^{-1}\right)=\tau(x)\tau(y)\tau^2(x)^{-1}.\] 
Here $\tau$ is called metahomomorphism on group $G$, 
please see \cite{Ding2006}, \cite{Ding2009} and \cite{Castelli2019} for more informations. 
If this solution is non-degenerate and $ \tau^2={\rm id}\neq \tau$, then it is a near-rack solution. 
\end{example}

\begin{example}
$(X, r)$ is  a near-rack solution with $X=[1, 2n]$ and 
\begin{align*}
b+2a-2&=2ng+d,\quad g\in\Bbb{N},\quad 0\leq d\leq 2n-1,\\
2n+1-b+2a-2&=2ne+f,\quad e\in\Bbb{N},\quad 0\leq f\leq 2n-1,\\
r(a,b)&=\left\{\begin{array}{ll}
(b,1), &a=1,\\
(2n, 2n-a+2), &a>1, d=0, 2\mid (a+b),\\
(d, 2n-a+2), &a>1, d>0, 2\mid (a+b),\\
(1, 2n-a+2), &a>1, f=0, 2\nmid (a+b),\\
(2n+1-f, 2n-a+2), &a>1, f>0, 2\nmid (a+b).
\end{array}\right.
\end{align*}
The Yetter-Drinfeld module $\mathcal K_{jk,p}^s$ in \cite{shi2023_NearRack} and 
\cite{Shi2020even}
is associated to this
near-rack solution $(X, r)$. Furthermore, 
$\mathcal K_{jk,p}^s$ is t-equivalent to a braided vector space of  type dihedral rack $\Bbb D_{2n}$.
\end{example}

\begin{example} 
Let $X=[1, 2n+1]$, $L^\gamma$ and $R^\gamma$ be two maps from $X\times X$
to itself such that
\begin{align*}
R^\gamma(a,b)&=\left\{\begin{array}{ll}
(b+\gamma, 2n-a+2), &b+\gamma\leq 2n+1,\\
(2n+1, 2n-a+2), &b+\gamma=2n+2,\\
(4n+3-\gamma-b, 2n-a+2), &b+\gamma\geq 2n+3,
\end{array}\right.\\
L^\gamma(a,b)&=\left\{\begin{array}{ll}
(b-\gamma, 2n-a+2), &\gamma<b,\\
(\gamma-b+1, 2n-a+2), &b\leq \gamma\leq b+2n.
\end{array}\right.
\end{align*}
Then $(X, r)$ is a near-rack solution with $r$ described as follows:
\begin{align*}
r(a, b)=\left\{\begin{array}{ll}
(b, 2n-a+2), &a=n+1,\\
L^{2(n-a+1)}(a, b), &a<n+1, 2\mid (a+b),\\
L^{2(a-n-1)}(a, b), &a>n+1, 2\nmid (a+b),\\
R^{2(n-a+1)}(a, b), &a<n+1, 2\nmid (a+b),\\
R^{2(a-n-1)}(a, b), &a>n+1, 2\mid (a+b).
\end{array}\right. 
\end{align*}
The Yetter-Drinfeld module $\mathcal N_{k,pq}^s$ appeared in \cite{Shi2020odd} and 
\cite{shi2023_NearRack} is associated to this solution $(X, r)$ and $\mathcal N_{k,pq}^s$
is t-equivalent to a braided vector space of type dihedral rack $\Bbb D_{2n+1}$. 
\end{example}

\begin{lemma}
If $(X, r)$ is a near-rack solution, then $r(x,y)=(x, y)$ if and only if 
$\sigma_x\tau(x)=x$ and $y=\tau(x)$.
\end{lemma}
\begin{proof}
$r(x,y)=(\sigma_x(y), \tau(x))=(x, y)$ if and only if $\sigma_x(y)=x$ and $y=\tau(x)$. 
\end{proof}

\begin{corollary}\label{FixedPoints}
If $(X, r)$ is a near-rack solution and $r(x, \tau(x))=(x, \tau(x))$, then $r(\tau(x), x)=(\tau(x), x)$. 
\end{corollary}
\begin{proof}
The identity  $r(x, \tau(x))=(x, \tau(x))$ implies  $\sigma_x\tau(x)=x$. We have
\begin{align*}
r(\tau(x), x)=(\sigma_{\tau(x)}(x), \tau^2(x))
=(\tau\sigma_x\tau(x), x)
=(\tau(x), x). 
\end{align*}
\end{proof}

\begin{lemma}
If $(X, r)$ is a near-rack solution, then $\sigma_x=\sigma_{\sigma_x\tau(x)}$ for any $x\in X$. 
\end{lemma}
\begin{proof}
The formula
$\sigma_x\sigma_{\tau(x)}=\sigma_{\sigma_{x}\tau(x)}\sigma_{\tau(x)}$ implies that 
$\sigma_x=\sigma_{\sigma_x\tau(x)}$.
\end{proof}

\begin{corollary}
If $(X, r)$ is a near-rack solution with the property that
 $\sigma_x=\sigma_y$ if and only if $x=y$ for any $x, y\in X$,
then $\sigma_x\tau (x)=x$ for any $x\in X$ and 
\[
|\{(x, y)\mid r(x, y)=(x, y), x, y\in X\}|=|X|. 
\] 
\end{corollary}

\begin{proposition}\label{ConstructionNearRack}
There is a bijective correspondence between  near-rack solutions and
the set $\left\{(X, \rhd), \tau\in \Bbb S_{|X|}\right\}$, where $(X, \rhd)$ is a rack with 
\begin{align}
\tau^2={\rm id}\neq \tau, \qquad \tau(\tau(x)\rhd y)&=x\rhd \tau(y), \quad \forall x, y\in X. 
\end{align}
\end{proposition}

\begin{remark}
This Proposition provides an algorithm to enumerate near-rack solutions
according to enumeration of  racks. 
\end{remark}

\begin{proof}
If $(X, r)$  is a near-rack solution, then $r(x, y)=(\sigma_x(y), \tau(x))$ with $\tau^2={\rm id}\neq \tau$. 
There is a rack structure on $X$ with $x\rhd y=\tau\sigma_{\tau(x)}(y)$, which implies that 
\begin{align*}
\tau(\tau(x)\rhd y)
=\tau\tau\sigma_{\tau^2(x)}(y)
=\sigma_{x}(y)
=\tau\sigma_{\tau(x)}\tau(y)=x\rhd \tau(y). 
\end{align*}
If $(X, \rhd)$ is a rack with given conditions, we set $\tilde r(x, y)=(x\rhd y, x)$ and 
$r=(\tau\times {\rm id})\tilde r(\tau\times {\rm id})$. Then we have  
\begin{align*}
r(x, y)&=(\tau\times {\rm id})\tilde r(\tau\times {\rm id})(x, y)
=(\tau(\tau(x)\rhd y), \tau(x))\\
&=(x\rhd \tau(y), \tau(x))\\
&=({\rm id}\times \tau)(x\rhd \tau(y), x)\\
&=({\rm id}\times \tau)\tilde r({\rm id}\times \tau)(x, y). 
\end{align*}
Set $T=\tau\times {\rm id}\times \tau$, then $T^2={\rm id}\times {\rm id}\times {\rm id}$,
$r\times {\rm id}=T(\tilde r\times {\rm id})T$ and ${\rm id}\times r=T({\rm id}\times \tilde r)T$
which implies that $(X, r)$ is a  solution of the Yang-Baxter equation. 
\end{proof}

Up to isomorphism, we list  examples of non-involutive near-rack solutions 
associated to racks listed in \cite[Table 9.1]{Heckenberger2015} or 
\cite[Table 17.1]{Heckenberger[2020]copyright2020} according to Proposition \ref{ConstructionNearRack}. 
Computations of the following examples are implemented by the software GAP. 
Some examples of t-equivalence related with those near-rack solutions and racks are presented in Appendix.  
\begin{example}
There is exactly one near-rack solution whose derived solution provided by the 
dihedral rack $\Bbb D_3$: $\sigma_1={\rm id}$, $\sigma_2=(1,3,2)$,  $\sigma_3=(1,2,3)$, 
$\tau=(2,3)\in\Bbb S_3$. 
\end{example}

\begin{example}
Fix an order of the conjugate class of $(1, 2) \in \Bbb S_4$ as 
$\{(1,2)$,  $(1,3)$,  $(1,4)$, $(2,3)$, $(2,4)$,$(3,4)\}$ and identify it as $[1,6]$,
 then the rack can be represented by 
 $\begin{pmatrix}
 1& 4& 5& 2& 3& 6 \\
4& 2& 6& 1& 5& 3 \\ 
5& 6& 3& 4& 1& 2 \\ 
2& 1& 3& 4& 6& 5 \\ 
3& 2& 1& 6& 5& 4 \\
1& 3& 2& 5& 4& 6
 \end{pmatrix}$.
There are exactly two  near-rack solutions whose derived solution provided by the
conjugate class of $(1,2)\in\Bbb S_4$: 
\begin{enumerate}
\item $\sigma_1={\rm id}$, $\sigma_2=(1,4,2)(3,5,6)$, $\sigma_3=(1,5,3)(2,4,6)$, 
$\sigma_4=(1,2,4)(3,6,5)$, $\sigma_5=(1,3,5)(2,6,4)$, $\sigma_6=(2,5)(3,4)$, 
$\tau=(2,4)(3,5)$; 
\item $\sigma_1=(2,3)(4,5)$, $\sigma_2=(1,4,6,3)(2,5)$, $\sigma_3=(1,5,6,2)(3,4)$, 
$\sigma_4=(1,2,6,5)(3,4)$,  $\sigma_5=(1,3,6,4)(2,5)$, $\sigma_6=(2,4)(3,5)$, 
$\tau=(2,5)(3,4)$. 
\end{enumerate}
\end{example}

\begin{example}
Fix an order of the conjugate class of $(1, 2, 3)$ in the alternating group ${\rm Alt}_4$ as
$\{(1,2,3), (1,3,4), (1,4,2), (2,4,3)\}$ and identify it as $[1,4]$, then the rack can be represented by
$\begin{pmatrix}
1& 3& 4& 2 \\
4& 2& 1& 3 \\
2& 4& 3& 1 \\
3& 1& 2& 4
\end{pmatrix}$. 
There is exactly one  near-rack solution whose derived solution provided by the
conjugate class of $(1,2,3)\in {\rm Alt}_4$:
$\sigma_1=(1,2,4)$, $\sigma_2=(1,3,2)$, $\sigma_3=(2,3,4)$, $\sigma_4=(1,4,3)$, $\tau=(1,4)(2,3)$.
\end{example}

\begin{example}
There are exactly two near-rack solutions whose derived solution provided by the
conjugate class of $(1,2,3,4)\in\Bbb S_4$ \cite[Example 17.2.5]{Heckenberger[2020]copyright2020}:
\begin{enumerate}
\item $\sigma_1=(2,3,5,4)$, $\sigma_2=(1,3)(2,5)(4,6)$, $\sigma_3=(1,5)(2,6)(3,4)$, 
$\sigma_4=(1,2)(3,4)(5,6)$, $\sigma_5=(1,4)(2,5)(3,6)$, $\sigma_6=(2,4,5,3)$, 
$\tau=(2,5)(3,4) \in \Bbb S_6$;
\item $\sigma_1=(1,2,3)(4,6,5)$, $\sigma_2=(1,6)(2,5)$, $\sigma_3=(1,3,5)(2,6,4)$, \\
$\sigma_4=(1,3,2)(4,5,6)$, $\sigma_5=(2,5)(3,4)$, $\sigma_5=(1,5,3)(2,4,6)$,\\ 
$\tau=(1,3)(2,5)(4,6)\in \Bbb S_6$.
\end{enumerate}
\end{example}

\begin{example}
There is exactly one near-rack solution whose derived solution provided by the 
affine rack ${\rm Aff}(5,2)$: 
$\sigma_1=(2,4,5,3)$, $\sigma_2=(1,5,2,3)$, $\sigma_3=(1,4,3,5)$, $\sigma_4=(1,3,4,2)$, 
$\sigma_5=(1,2,5,4)$, $\tau=(2,5)(3,4)\in\Bbb S_5$.  
\end{example}

\begin{example}
There is exactly one near-rack solution whose derived solution provided by the 
affine rack ${\rm Aff}(5,3)$: 
$\sigma_1= (2,3,5,4)$, $\sigma_2=(1,4,5,2)$, $\sigma_3=(1,2,4,3)$, 
$\sigma_4=(1,5,3,4)$, $\sigma_5=(1,3,2,5)$, $\tau=(2,5)(3,4)\in\Bbb S_5$. 
\end{example}

\begin{example}
There is exactly one near-rack solution whose derived solution provided by the 
affine rack ${\rm Aff}(7,3)$: 
 $\sigma_1=(2,5,3)(4,6,7)$,  $\sigma_2=(1,6,5)(2,3,7)$,  $\sigma_3=(1,4,2)(3,5,6)$, 
  $\sigma_4=(1,2,6)(4,7,5)$,  $\sigma_5=(1,7,3)(2,4,5)$, \\
$\sigma_6=(1,5,7)(3,6,4)$, $\sigma_7=(1,3,4)(2,7,6)$, $\tau=(2,7)(3,6)(4,5)\in\Bbb S_7$.
\end{example}

\begin{example}
There is exactly one near-rack solution whose derived solution provided by the 
affine rack ${\rm Aff}(7,5)$: 
$\sigma_1=(2,3,5)(4,7,6)$, $\sigma_2=(1,4,3)(2,6,7)$, $\sigma_3=(1,7,5)(3,4,6)$, 
$\sigma_4=(1,3,7)(2,5,4)$, $\sigma_5=(1,6,2)(4,5,7)$,\\
$\sigma_6=(1,2,4)(3,6,5)$, 
$\sigma_7=(1,5,6)(2,7,3)$, $\tau=(2,7)(3,6)(4,5)\in\Bbb S_7$.   
\end{example}

\begin{example}
Fix an order of the conjugate class of $(1, 2)$ in $\Bbb S_5$ as
\[\{(1,2),(1,3),(1,4),(1,5),(2,3),(2,4),(2,5),(3,4),(3,5),(4,5)\}, 
\]
and identify it as $[1,10]$, 
then the rack can be represented by 
 \[
 \begin{pmatrix}
  1&5&6&7&2&3&4&8&9&10\\
 5&2&8&9&1&6&7&3&4&10\\
 6&8&3&10&5&1&7&2&9&4\\ 
 7&9&10&4&5&6&1&8&2&3\\
 2&1&3&4&5&8&9&6&7&10\\
 3&2&1&4&8&6&10&5&9&7\\
 4&2&3&1&9&10&7&8&5&6\\
 1&3&2&4&6&5&7&8&10&9\\
 1&4&3&2&7&6&5&10&9&8\\
 1&2&4&3&5&7&6&9&8&10
 \end{pmatrix}.
 \]
There are exactly two  near-rack solutions whose derived solution provided by the
conjugate class of $(1,2)\in\Bbb S_5$: 
\begin{enumerate}
\item $\sigma_1={\rm id}$, $\sigma_2=(1,5,2)(3,6,8)(4,7,9)$,  $\sigma_3=(1,6,3)(2,5,8)(4,7,10)$, \\
          $\sigma_4=(1,7,4)(2,5,9)(3,6,10)$, $\sigma_5=(1,2,5)(3,8,6)(4,9,7)$, \\
          $\sigma_6=(1,3,6)(2,8,5)(4,10,7)$, $\sigma_7=(1,4,7)(2,9,5)(3,10,6)$, \\
          $\sigma_8=(2,6)(3,5)(4,7)(9,10)$, $\sigma_9=(2,7)(3,6)(4,5)(8,10)$, \\
          $\sigma_{10}=(2,5)(3,7)(4,6)(8,9)$, 
         $\tau=(2,5)(3,6)(4,7)$;
\item $\sigma_1=(3,4)(6,7)(8,9)$, $\sigma_2=(1,5,2)(3,7,8,4,6,9)$, \\
          $\sigma_3=(1,6,10,4)(2,5,8,9)(3,7)$, 
          $\sigma_4=(1,7,10,3)(2,5,9,8)(4,6)$, \\$\sigma_5=(1,2,5)(3,9,6,4,8,7)$, 
          $\sigma_6=(1,3,10,7)(2,8,9,5)(4,6)$, \\$\sigma_7=(1,4,10,6)(2,9,8,5)(3,7)$, 
          $\sigma_8=(2,6,4,5,3,7)(8,10,9)$, \\$\sigma_9=(2,7,3,5,4,6)(8,9,10)$, $\sigma_{10}=(2,5)(3,6)(4,7)$, \\
          $\tau=(2,5)(3,7)(4,6)(8,9)$.
\end{enumerate}
\end{example}

\begin{proposition}\label{All_Involutive_NearRacks}
Let  $X=\{1, 2, \cdots, |X|\}$ and  $(X, r)$ be an involutive near-rack solution. Then 
$(X, r)$ is isomorphic to a solution with  
\begin{equation}\label{InvolutiveSigma}
\sigma_x=\tau=\prod_{i=1}^k(2i-1, 2i)\in \Bbb S_{|X|}, \qquad \forall x\in X, 
\end{equation}
where  $k$ is  an integer with  $1\leq k\leq \frac{|X|}2$. 
\end{proposition}

\begin{proof}
For any $x, y\in X$, we have 
\begin{align*}
(x, y)&=r^2(x, y)=r(\sigma_x(y), \tau(x))=(\sigma_{\sigma_x(y)}\tau(x), \tau\sigma_x(y)),  
\end{align*}
which implies that $y=\tau\sigma_x(y)$. So $\sigma_x=\tau$ for any $x\in X$. 
Up to conjugation, we can let $\tau$ be the permutation presented in \eqref{InvolutiveSigma}. 
\end{proof}

\section{Near-rack solutions and t-equivalence}
Let $(X, r)$ be a near-rack solution and $W_{X, r}=\bigoplus_{i\in X}\k w_i$ be a braided vector 
space with the braiding given by 
\begin{align}\label{YBEquation}
c(w_{i}\otimes w_j)
&=R_{i,j} w_{\sigma_i(j)}\otimes w_{\tau(i)}, \quad 
\text{where}\,\, R_{i,j}\in\k^\times\,\,\text{and} \\ 
R_{i, j}R_{\tau(i), k}R_{\sigma_i(j),\sigma_{\tau(i)}(k)}
&=R_{j,k}R_{i,\sigma_j(k)}R_{\tau(i),\tau(j)}, 
\quad \forall i, j, k\in X. \label{YBEquation}
\end{align}

\begin{lemma}
Let $(X, r)$ be a near-rack solution and $r(x, \tau(x))=(x, \tau(x))$ for some $x\in X$, then 
$R_{x, \tau(x)}=R_{\tau(x), x}$. 
\end{lemma}
\begin{proof}
It is a direct result from Corollary \ref{FixedPoints} and  Formula \eqref{YBEquation}. 
\end{proof}

\begin{theorem}\label{KeyTheorem}
Let $(X, r)$ be a near-rack solution. 
Define an invertible map $\varphi: W_{X, r}\to W_{X, r}$ such that $\varphi(w_i)=z_i w_{\tau(i)}$ with $z_i\in\k^{\times}$ for any $i\in X$. Then 
$\tilde c=(\varphi^{-1} \otimes {\rm id})c(\varphi \otimes {\rm id})
=({\rm id}\otimes \varphi^{-1})c({\rm id}\otimes \varphi)
$ if and only if 
\begin{equation}\label{T-EquEq}
z_i^2=z_jz_{\sigma_i\tau(j)}\frac{R_{i, \tau(j)}}{R_{\tau(i), j}},\qquad \forall i, j\in X.
\end{equation}
In other words, if there exist non-zero parameters $z_i$ for $i\in X$ satisfying the equations 
\eqref{T-EquEq}, then $(W_{X, r}, c)$ is t-equivalent to $(W_{X, \tilde r}, \tilde c)$ which is of rack type, 
where $(X, \tilde r)$ is the  derived solution of $(X, r)$   and 
 $W_{X, r}=W_{X, \tilde r}$ as vector spaces.  
\end{theorem}

\begin{remark}
If  there exist non-zero parameters $z_i$ for $i\in X$ satisfying the equations 
\eqref{T-EquEq}, then 
\begin{align*}
z_i^{2|X|}&=\prod_{j\in X}z_jz_{\sigma_i\tau(j)}\frac{R_{i, \tau(j)}}{R_{\tau(i), j}}
=\prod_{j\in X}z_j^2, \qquad \text{if  } i=\tau(i)\in X; \\
(z_iz_{\tau(i)})^{2|X|}&=\prod_{j\in X}
z_jz_{\sigma_i\tau(j)}\frac{R_{i, \tau(j)}}{R_{\tau(i), j}}
z_jz_{\sigma_{\tau(i)}\tau(j)}\frac{R_{\tau(i), \tau(j)}}{R_{i, j}}
=\prod_{j\in X}z_j^4, \qquad \forall i\in X. 
\end{align*}
\end{remark}

\begin{proof}
Set $\bar c=(\varphi^{-1} \otimes {\rm id})c(\varphi \otimes {\rm id})$
and  $\tilde c=({\rm id}\otimes \varphi^{-1})c({\rm id}\otimes \varphi)$, then 
\begin{align*}
\bar c(w_i\otimes w_j)&=(\varphi^{-1} \otimes {\rm id})c(\varphi \otimes {\rm id})(w_i\otimes w_j)\\
&=z_i(\varphi^{-1} \otimes {\rm id})c(w_{\tau(i)}\otimes w_j)\\
&=z_iR_{\tau(i), j}(\varphi^{-1} \otimes {\rm id})(w_{\sigma_{\tau(i)}(j)}\otimes w_i)\\
&=z_iR_{\tau(i), j}z_{\tau\sigma_{\tau(i)}(j)}^{-1}w_{\tau\sigma_{\tau(i)}(j)}\otimes w_i,\\
\tilde c(w_i\otimes w_j)&=({\rm id}\otimes \varphi^{-1})c({\rm id}\otimes \varphi)(w_i\otimes w_j)\\
&=z_j({\rm id}\otimes \varphi^{-1})c(w_i\otimes w_{\tau(j)})\\
&=z_jR_{i, \tau(j)}({\rm id}\otimes \varphi^{-1})(w_{\sigma_i\tau(j)}\otimes w_{\tau(i)})\\
&=z_jR_{i, \tau(j)}z_i^{-1}w_{\sigma_i\tau(j)}\otimes w_i\\
&=z_jR_{i, \tau(j)}z_i^{-1}w_{\tau\sigma_{\tau(i)}(j)}\otimes w_i.
\end{align*}
So $\bar c=\tilde c$ if and only if 
$
z_jR_{i, \tau(j)}z_i^{-1}=z_iR_{\tau(i), j}z_{\tau\sigma_{\tau(i)}(j)}^{-1}$
for any  $i, j\in X$. 
\end{proof}

\begin{corollary}\label{Involutive_T_equivalent}
Let  $(X, r)$ be an involutive  near-rack solution and $\varphi: W_{X, r}\mapsto W_{X, r}$ be an invertible map with 
$\varphi(w_i)=z_i w_{\tau(i)}$, where  
\begin{equation}\label{InvolutiveCase_Parameters}
z_i=\sqrt{\frac{R_{i,\tau(1)}}{R_{\tau(i),1}}}, \forall i\in X.
\end{equation}
Then 
$
\tilde c=(\varphi^{-1} \otimes {\rm id})c(\varphi \otimes {\rm id})
=({\rm id}\otimes \varphi^{-1})c({\rm id}\otimes \varphi), 
$ 
hence $(W_{X, r}, c)$ is t-equivalent to $(W_{X, \tilde r}, \tilde c)$. 
\end{corollary}

\begin{remark}
$\tilde c(w_i\otimes w_j)=\frac{z_i}{z_j}R_{\tau(i),j} w_j\otimes w_i$ for any $i, j\in X$. 
\end{remark}

\begin{proof}
We only need to check that \eqref{InvolutiveCase_Parameters} is a solution of the equations 
\eqref{T-EquEq}:
\begin{align*}
\frac{z_i^2}{z_j^2}&=\frac{R_{i,\tau(1)}}{R_{\tau(i),1}}\cdot \frac{R_{\tau(j),1}}{R_{j,\tau(1)}}
\xlongequal{\eqref{YBEquation}} \frac{R_{i, \tau(j)}}{R_{\tau(i), j}}, \quad \forall i, j\in X.
\end{align*}
\end{proof}

\section{Finite-dimensional Nichols algebras associated to involutive near-rack solutions}

\begin{lemma}
Let $(X, r)$ be an involutive near-rack solution, $(W_{X, r}, c)$ be t-equivalent to 
$(W_{X, \tilde r}, \tilde c)$ as given in Corollary \ref{Involutive_T_equivalent}. 
Denote $q_{ij}=\frac{z_i}{z_j}R_{\tau(i),j}$ and $\tilde{q}_{ij}=q_{ij}q_{ji}$, then 
\begin{align*}
q_{ii}=q_{\tau(i)\tau(i)},\quad 
\tilde{q}_{ij}=\tilde{q}_{\tau(i)\tau(j)},\quad \forall i, j\in X. 
\end{align*} 
\end{lemma}

\begin{proof}
$q_{ii}=R_{\tau(i),i}=R_{i, \tau(i)}=q_{\tau(i)\tau(i)}$. According to  Formula \eqref{YBEquation}, we have
\[
R_{i,\tau(j)}R_{\tau(i), i}R_{j,\tau(i)}
=R_{\tau(j),i}R_{i,\tau(i)}R_{\tau(i),j},
\]
which implies $R_{i,\tau(j)}R_{j,\tau(i)}=R_{\tau(j),i}R_{\tau(i),j}$. Then 
\[
\tilde{q}_{\tau(i)\tau(j)}=R_{i,\tau(j)}R_{j,\tau(i)}=R_{\tau(i),j}R_{\tau(j),i}=\tilde{q}_{ij}.
\]
\end{proof}

\begin{example}\cite{Andruskiewitsch2018, Shi2020even, Shi2023, shi2023_NearRack}
Let $(X, r)$ be an involutive near-rack solution with $r(i, j)=(\tau(j), \tau(i))$, 
where $\tau=(1,2)\in\Bbb S_2$. 
Then $W_{X, r}=\bigoplus_{p=1}^2\k w_p$ is a braided vector space with a braiding given by 
\begin{align*}
c(w_1\otimes w_1)&=a w_2\otimes w_2,\quad 
&c(w_1\otimes w_2)&=b   w_1\otimes w_2,\\
c(w_2\otimes w_1)&=b w_2\otimes w_1,\quad 
&c(w_2\otimes w_2)&=e   w_1\otimes w_1.
\end{align*}
Define an invertible map $\varphi: W_{X, r}\to W_{X, r}$ via $\varphi(w_1)=w_2$, 
$\varphi(w_2)=\sqrt{\frac ea} w_1$, then 
$\tilde c=(\varphi^{-1}\otimes {\rm id})c(\varphi\otimes {\rm id})
=({\rm id}\otimes \varphi^{-1})c({\rm id}\otimes \varphi)$. 
So $(W_{X, r}, c)$ is t-equivalent to  $(W_{X, \tilde r}, \tilde c)$. The braided vector space  
$(W_{X, \tilde r}, \tilde c)$ is of diagonal type with
$q_{11}=q_{22}=b$,  $\tilde q_{12}=ae$, which implies 
\[
\dim \mathfrak{B}(W_{X, \tilde r}, \tilde c)=
\left\{\begin{array}{ll}
27, & ae=b^2, b^3=1\neq b, \quad \text{Cartan type $A_2$},\\
4m, &b=-1, ae\in\Bbb{G}_m^\prime,  m\geq 2, \quad 
\text{super type ${\bf A}_{2}(q;\{1,2\})$ },\\
m^2, &ae=1, b\in\Bbb{G}_m^\prime\,\,\text{for}\,\,m\geq 2, \quad \text{Cartan type $A_1\times A_1$},\\
\infty, & otherwise.
\end{array}\right.
\]
\end{example}

\begin{example}
Let $(X, r)$ be an involutive near-rack solution with $r(i, j)=(\tau(j), \tau(i))$, 
where $\tau=(1,2)\in\Bbb S_3$. 
Then $W_{X, r}=\bigoplus_{p=1}^3\k w_p$ is a braided vector space with a braiding given by 
\[
c(w_i\otimes w_j)=x_{3(i-1)+j} w_{\tau(j)}\otimes w_{\tau(i)}, \quad \forall i, j\in[1,3],
\]
where $x_k\in\Bbbk^\times$ for $k\in[1,9]$ and $x_4=x_2, x_8=x_6x_7x_3^{-1}, x_5=x_1x_6^2x_3^{-2}$.  
Define an invertible map $\varphi: W_{X, r}\to W_{X, r}$ via $\varphi(w_1)=z_1 w_2$, 
$\varphi(w_2)=z_2 w_1$, $\varphi(w_3)=z_3 w_3$, where $z_1z_2z_3\in \Bbbk^\times$, and
$z_2=\frac{x_6z_1}{x_3}$,  $z_3=z_1\left(\frac{x_6}{x_3}\right)^{\frac12}$.
Then $\tilde c=(\varphi^{-1}\otimes {\rm id})c(\varphi\otimes {\rm id})
=({\rm id}\otimes \varphi^{-1})c({\rm id}\otimes \varphi)$. So $(W_{X, r}, c)$ is t-equivalent to 
$(W_{X, \tilde r}, \tilde c)$. The braided vector space  $(W_{X, \tilde r}, \tilde c)$ is of diagonal type 
with 
\[
q_{11}=q_{22}=x_2,\quad q_{33}=x_9,\quad \tilde q_{12}=\frac{(x_1x_6)^2}{x_3^2}, \quad
\tilde q_{13}=\tilde q_{23}=x_6x_7.
\]
If $(x_1x_6)^2\neq x_3^2$ and $x_6x_7\neq 1$, then the  generalized Dynkin diagram of $(W_{X, r}, \tilde c)$ is given by 
$$
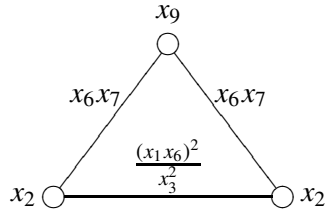
\xy 
(0,0)*\cir<4pt>{}="E1", 
(15,20)*\cir<4pt>{}="E2",
(30,0)*\cir<4pt>{}="E3",
(-4,0)*+{x_2},
(15,24)*+{x_9},
(15,4)*+{\frac{(x_1x_6)^2}{x_3^2}},
(34,0)*+{x_2},
(5.5,13)*+{x_6x_7},
(24.5,13)*+{x_6x_7},
\ar @{-}"E1";"E2"
\ar @{-}"E2";"E3"
\ar @{-}"E1";"E3"
\endxy
$$
\end{example}

\begin{example}
Let $(X, r)$ be an involutive near-rack solution with $r(i, j)=(\tau(j), \tau(i))$, 
where $\tau=(1,2)(3,4)\in\Bbb S_4$. 
Then $W_{X, r}=\bigoplus_{p=1}^4\k w_p$ is a braided vector space with a braiding given by 
\[
c(w_i\otimes w_j)=x_{3(i-1)+j} w_{\tau(j)}\otimes w_{\tau(i)}, \quad \forall i, j\in[1,4],
\]
where $x_k\in\Bbbk^\times$ for $k\in[1, 16]$ and 
\begin{align*}
x_5&=x_2, & x_{15}&=x_{12}, &x_{14}&=x_8x_9x_3^{-1},\\
x_{13}&=x_4x_{10}x_7^{-1}, &x_{16}&=x_4x_8x_{11}(x_3x_7)^{-1}, &
x_6&=x_1x_7x_8(x_3x_4)^{-1}. 
\end{align*}
Define an invertible map $\varphi: W_{X, r}\to W_{X, r}$ via $\varphi(w_1)=z_1 w_2$, 
$\varphi(w_2)=z_2 w_1$, $\varphi(w_3)=z_3 w_4$, $\varphi(w_4)=z_4 w_3$, 
where $z_1z_2z_3z_4\in \Bbbk^\times$, and
\begin{align*}
z_2=z_1\sqrt{\frac{x_7x_8}{x_3x_4}}, \quad z_3=z_1\sqrt{\frac{x_7}{x_4}}, \quad z_4=z_1\sqrt{\frac{x_8}{x_3}}.
\end{align*}
Then $\tilde c=(\varphi^{-1}\otimes {\rm id})c(\varphi\otimes {\rm id})
=({\rm id}\otimes \varphi^{-1})c({\rm id}\otimes \varphi)$. 
So $(W_{X, r}, c)$ is t-equivalent to  $(W_{X, \tilde r}, \tilde c)$. The braided vector space  
$(W_{X, \tilde r}, \tilde c)$ is of diagonal type with 
\begin{align*}
q_{11}&=q_{22}=x_2,  &
\tilde q_{13}&=\tilde q_{24}=x_4x_{10},&
\tilde q_{12}&=\frac{x_1^2x_7x_8}{x_3x_4},\\
q_{33}&=q_{44}=x_{12},  &
\tilde q_{14}&=\tilde q_{23}=x_8x_{9}, &
\tilde q_{34}&=\frac{x_{11}^2x_4x_8}{x_3x_7}.
\end{align*}
If $x_4x_{10}\neq 1$, $x_1^2x_7x_8\neq x_3x_4$, $x_8x_9\neq 1$, $x_{11}^2x_4x_8\neq x_3x_7$, 
then the  generalized Dynkin diagram of $(W_{X, \tilde r}, \tilde c)$ is given by
$$
\xy 
(0,0)*\cir<4pt>{}="E1", 
(30,0)*\cir<4pt>{}="E2",
(0,30)*\cir<4pt>{}="E3",
(30,30)*\cir<4pt>{}="E4",
(-5,0)*+{x_{12}},
(35,0)*+{x_{12}},
(-4,30)*+{x_2},
(34,30)*+{x_2},
(-5,15)*+{x_4x_{10}},
(35,15)*+{x_4x_{10}},
(7.5,12)*+{x_8x_9},
(23,12)*+{x_8x_9},
(15,26)*+{\frac{x_1^2x_7x_8}{x_3x_4}},
(15,4)*+{\frac{x_{11}^2x_4x_8}{x_3x_7}},
\ar @{-}"E1";"E2"
\ar @{-}"E2";"E3"
\ar @{-}"E1";"E3"
\ar @{-}"E2";"E4"
\ar @{-}"E1";"E4"
\ar @{-}"E3";"E4"
\endxy
$$
\end{example}

\begin{definition}\cite[Page 411]{Andruskiewitsch2017}
Let ${\bf A}_n(q;\Bbb J)$ be the generalized generalized Dynkin diagram 
$$\xy 
(0,0)*\cir<2pt>{}="E1", 
(23,0)*\cir<2pt>{}="E2",
(46,0)*\cir<2pt>{}="E3",
(69,0)*\cir<2pt>{}="E4",
(0,2.5)*+{q_{11}},
(11.5,2.5)*+{\tilde q_{12}},
(23,2.5)*+{q_{22}},
(34.5,-0.1)*+{\cdots\cdots\cdots\cdots},
(46,2.5)*+{q_{n-1n-1}},
(58,2.5)*+{\tilde q_{n-1n}},
(69,2.5)*+{\tilde q_{nn}},
\ar @{-}"E1";"E2"
\ar @{-}"E3";"E4"
\endxy,$$
where the labels satisfy the following requirements:
\begin{enumerate}
\item $q=q_{nn}^2\tilde q_{n-1n}\in \Bbbk^{\times}-\{\pm 1\}$;
\item if $i\in \Bbb J$, then $q_{ii}=-1$ and $\tilde q_{i-1 i}=\tilde q_{ii+1}^{-1}$;
\item if $i\notin \Bbb J$, then $\tilde q_{i-1i}=q_{ii}^{-1}
=\tilde q_{ii+1}$ (only the second equality if $i=1$, only the first if $i=n$).
\end{enumerate}
\end{definition}

\begin{remark}
${\bf A}_n(q;\Bbb J)$ is of Cartan type $A_n$ in case that $\Bbb J$ is an empty set,  otherwise 
it is of super type ${\bf A}_n$. 
\end{remark}

\begin{definition}
We say ${\bf A}_n(q;\Bbb J)$ is of symmetric super type if the following additional conditions are satisfied. 
\begin{enumerate}
\item $\Bbb J$ is not an empty set;
\item $q_{kk}=q_{(n+1-k)(n+1-k)}$, $\tilde q_{kk-1}=\tilde q_{(n-k)(n+1-k)}$ for  $1\leq k\leq \frac n2$;
\item if $n=2m$ is even and $m\notin \Bbb J$, then $\tilde q_{mm+1}=q_{mm}^2$, 
$q_{mm}=q_{m+1m+1}\in\Bbb G_3^\prime$. 
\end{enumerate}
\end{definition}

\begin{proposition}\label{LengthTwo}
Let $(X, r)$ with $r(i, j)=(\tau(j), \tau(i))$ be an involutive near-rack solution, 
where $\tau=(1, 2)\in \Bbb S_n$ with $n\geq 3$. 
Braided vector spaces $(W_{X, r}, c)$ and $(W_{X,\tilde r}, \tilde c)$ are t-equivalent as described in 
Corollary \ref{Involutive_T_equivalent}. 
If $\mathfrak{B}(W_{X, \tilde r}, \tilde c)$ 
is finite-dimensional and its generalized Dynkin diagram is connected, then $(W_{X, \tilde r}, \tilde c)$ is either one of 
the following cases:
\begin{enumerate}
\item  Cartan type $D_n$; 
\item   
$
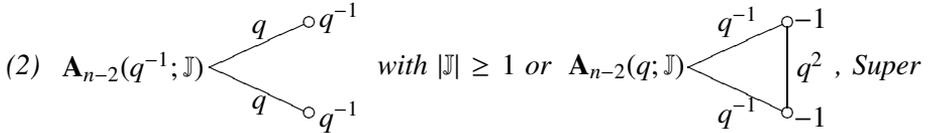
\xy 
(0,0)*\cir<0pt>{}="E1", 
(13,-6)*\cir<2pt>{}="E2",
(13,6)*\cir<2pt>{}="E3",
(-10,0)*+{{\bf A}_{n-2}(q^{-1}; \Bbb J)},
(6.5,5)*+{q},
(6.5,-5)*+{q},
(17,6.5)*+{q^{-1}},
(17,-6.5)*+{q^{-1}},
\ar @{-}"E1";"E2"
\ar @{-}"E1";"E3"
\endxy$ with $|\Bbb J|\geq 1$ or
$\xy 
(0,0)*\cir<0pt>{}="E1", 
(13,-6)*\cir<2pt>{}="E2",
(13,6)*\cir<2pt>{}="E3",
(-8,0)*+{{\bf A}_{n-2}(q; \Bbb J)},
(6.5,6)*+{q^{-1}},
(6.5,-6)*+{q^{-1}},
(16,6.5)*+{-1},
(16,0)*+{q^2},
(16,-6.5)*+{-1},
\ar @{-}"E1";"E2"
\ar @{-}"E1";"E3"
\ar @{-}"E2";"E3"
\endxy$, 
Super type ${\bf D}_n$ with $n\geq 4$, 
see \cite[Page 425]{Andruskiewitsch2017} and Table 4 in \cite{heckenberger2009classification}; 
\item $
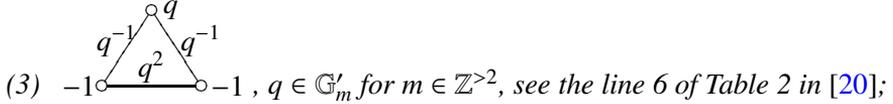
\xy 
(0,0)*\cir<2pt>{}="E1", 
(13,0)*\cir<2pt>{}="E2",
(6.5,10)*\cir<2pt>{}="E3",
(-3,0)*+{-1},
(6.5,2.5)*+{q^2},
(9,10)*+{q},
(16.5,0)*+{-1},
(2,6)*+{q^{-1}},
(13,6)*+{q^{-1}},
\ar @{-}"E1";"E2"
\ar @{-}"E2";"E3"
\ar @{-}"E1";"E3"
\endxy$,
$q\in\Bbb G^{\prime}_m$ for $m\in\Bbb Z^{> 2}$, see 
the line 6 of Table 2 in \cite{heckenberger2009classification};
\item Cartan type $A_3$;
\item Symmetric super type ${\bf A}_3(q;\Bbb J)$;
\item 
$\xy 
(0,0)*\cir<2pt>{}="E1", 
(13,0)*\cir<2pt>{}="E2",
(26,0)*\cir<2pt>{}="E3",
(0,2.5)*+{q},
(6.5,2.5)*+{q^{-1}},
(13,2.5)*+{-1},
(19.5,2.5)*+{q^{-1}},
(26,2.5)*+{q},
\ar @{-}"E1";"E2"
\ar @{-}"E2";"E3"
\endxy$ or 
$\xy 
(0,0)*\cir<2pt>{}="E1", 
(13,0)*\cir<2pt>{}="E2",
(6.5,10)*\cir<2pt>{}="E3",
(-3,0)*+{-1},
(6.5,2.5)*+{q^2},
(16.5,0)*+{-1},
(2,6)*+{q^{-1}},
(13,6)*+{q^{-1}},
(10,10)*+{-1},
\ar @{-}"E1";"E2"
\ar @{-}"E2";"E3"
\ar @{-}"E1";"E3"
\endxy$, type ${\rm D}(2, 1)$, 
$q\in\Bbb G^{\prime}_m$ for $m\in\Bbb Z^{> 2}$, see 
\cite[Page 432]{Andruskiewitsch2017} and lines 9, $10$, $11$ of Table 2 in \cite{heckenberger2009classification}; 
\item 
$
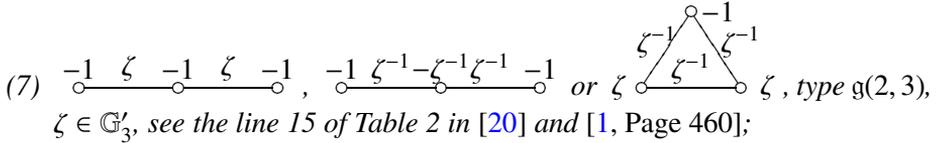
\xy 
(0,0)*\cir<2pt>{}="E1", 
(13,0)*\cir<2pt>{}="E2",
(26,0)*\cir<2pt>{}="E3",
(0,2.5)*+{-1},
(6.5,2.5)*+{\zeta},
(13,2.5)*+{-1},
(19.5,2.5)*+{\zeta},
(26,2.5)*+{-1},
\ar @{-}"E1";"E2"
\ar @{-}"E2";"E3"
\endxy$, 
$\xy 
(0,0)*\cir<2pt>{}="E1", 
(13,0)*\cir<2pt>{}="E2",
(26,0)*\cir<2pt>{}="E3",
(0,2.5)*+{-1},
(6.5,2.5)*+{\zeta^{-1}},
(13,2.5)*+{-\zeta^{-1}},
(19.5,2.5)*+{\zeta^{-1}},
(26,2.5)*+{-1},
\ar @{-}"E1";"E2"
\ar @{-}"E2";"E3"
\endxy$ or
$\xy 
(0,0)*\cir<2pt>{}="E1", 
(13,0)*\cir<2pt>{}="E2",
(6.5,10)*\cir<2pt>{}="E3",
(-3,0)*+{\zeta},
(6.5,2.5)*+{\zeta^{-1}},
(16.5,0)*+{\zeta},
(2,6)*+{\zeta^{-1}},
(13,6)*+{\zeta^{-1}},
(10,10)*+{-1},
\ar @{-}"E1";"E2"
\ar @{-}"E2";"E3"
\ar @{-}"E1";"E3"
\endxy$,  type $\mathfrak{g}(2, 3)$, 
$\zeta\in \Bbb G_3^\prime$, 
see the line 15 of Table 2 in \cite{heckenberger2009classification} 
and \cite[Page 460]{Andruskiewitsch2017}; 
\item 
$\xy 
(0,0)*\cir<2pt>{}="E1", 
(13,0)*\cir<2pt>{}="E2",
(26,-6)*\cir<2pt>{}="E3",
(26,6)*\cir<2pt>{}="E4",
(0,-3)*+{\zeta^{-1}},
(6.5,2.5)*+{\zeta},
(13,-3)*+{\zeta},
(29.5,6)*+{-1},
(29.5,-6)*+{-1},
(19.5,-5.5)*+{\zeta^{-1}},
(19.5,5.5)*+{\zeta^{-1}},
\ar @{-}"E1";"E2"
\ar @{-}"E2";"E3"
\ar @{-}"E2";"E4"
\endxy$ or 
$\xy 
(0,0)*\cir<2pt>{}="E1", 
(13,0)*\cir<2pt>{}="E2",
(26,-6)*\cir<2pt>{}="E3",
(26,6)*\cir<2pt>{}="E4",
(0,-3)*+{\zeta^{-1}},
(6.5,2.5)*+{\zeta},
(13,-3)*+{\zeta^{-1}},
(29.5,6)*+{-1},
(29.5,-6)*+{-1},
(19.5,-5.5)*+{\zeta},
(19.5,5.5)*+{\zeta},
\ar @{-}"E1";"E2"
\ar @{-}"E2";"E3"
\ar @{-}"E2";"E4"
\endxy,$  type $\mathfrak{g}(3, 3)$, $\zeta\in\Bbb G_3^\prime$, 
see  \cite[Page 48]{Angiono2019a} and the line 18 of Table 3  in \cite{heckenberger2009classification}.
\end{enumerate}
\end{proposition}

\begin{proof}
Since the generalized Dynkin diagram of $(W_{X, \tilde r}, \tilde c)$ is connected, there exists some  $i\in [3, n]$ such that 
$i$ connected with vertexes $1$ and $2$.  Without loss of generality, we set $i=3$. 
Suppose there is some $j\in [4, n]$ such that $j$ is connected with $1$ and $2$, then 
the generalized Dynkin diagram of $(W_{X, \tilde r}, \tilde c)$ contains the subgraph  $
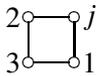
\xy 
(0,0)*\cir<2pt>{}="E1", 
(6,0)*\cir<2pt>{}="E2",
(0,6)*\cir<2pt>{}="E3",
(6,6)*\cir<2pt>{}="E4",
(-2,0)*+{3},
(8,0)*+{1},
(-2,6)*+{2},
(8,6)*+{j},
\ar @{-}"E1";"E2"
\ar @{-}"E1";"E3"
\ar @{-}"E2";"E4"
\ar @{-}"E3";"E4"
\endxy$, which is contradicted with that $\mathfrak{B}(W_{X, r}, c)$ is finite dimensional. 
According to Heckenberger's work \cite{heckenberger2009classification}, 
the generalized Dynkin diagram of $(W_{X, \tilde r}, \tilde c)$  looks like 
$$
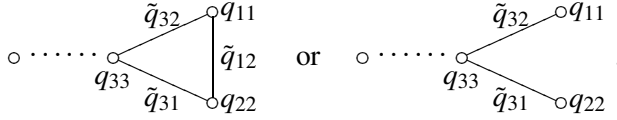
\xy 
(0,0)*\cir<2pt>{}="E1", 
(13,0)*\cir<2pt>{}="E2",
(26,-6)*\cir<2pt>{}="E3",
(26,6)*\cir<2pt>{}="E4",
(6.5,0)*+{\cdots\cdots},
(13,-3)*+{q_{33}},
(29.5,6)*+{q_{11}},
(29.5,0)*+{\tilde q_{12}},
(29.5,-6)*+{q_{22}},
(19.5,-5.5)*+{\tilde q_{31}},
(19.5,5.5)*+{\tilde q_{32}},
\ar @{-}"E2";"E3"
\ar @{-}"E2";"E4"
\ar @{-}"E3";"E4"
\endxy \text{\quad or \quad}  
\xy 
(0,0)*\cir<2pt>{}="E1", 
(13,0)*\cir<2pt>{}="E2",
(26,-6)*\cir<2pt>{}="E3",
(26,6)*\cir<2pt>{}="E4",
(6.5,0)*+{\cdots\cdots},
(13,-3)*+{q_{33}},
(29.5,6)*+{q_{11}},
(29.5,-6)*+{q_{22}},
(19.5,-5.5)*+{\tilde q_{31}},
(19.5,5.5)*+{\tilde q_{32}},
\ar @{-}"E2";"E3"
\ar @{-}"E2";"E4"
\endxy,$$ 
where $q_{11}=q_{22}$ and $\tilde q_{31}=\tilde q_{32}$. 
\end{proof}

\begin{proposition}
Let $(X, r)$ with $r(i, j)=(\tau(j), \tau(i))$ be an involutive near-rack solution, 
where  $\tau=(1, 2)(3,4)\in \Bbb S_n$ with $n\geq 5$. 
Braided vector spaces $(W_{X, r}, c)$ and $(W_{X, \tilde r}, \tilde c)$ are t-equivalent as described in 
Corollary \ref{Involutive_T_equivalent}. 
If $\mathfrak{B}(W_{X, \tilde r}, \tilde c)$ 
is finite-dimensional and its generalized Dynkin diagram is connected, then $(W_{X, \tilde r}, \tilde c)$ is either one of 
the following cases:
\begin{enumerate}
\item Cartan type $E_6$;
\item Cartan type $A_5$;
\item symmetric super type ${\bf A}_5(q;\Bbb J)$;
\item 
$
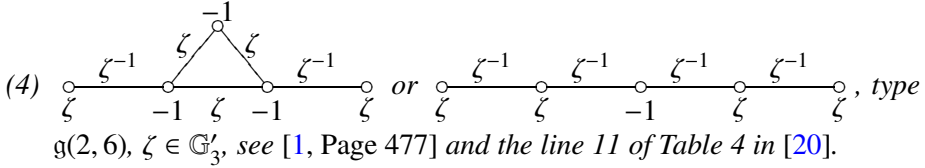
\xy 
(0,0)*\cir<2pt>{}="E1", 
(13,0)*\cir<2pt>{}="E2",
(26,0)*\cir<2pt>{}="E3",
(39,0)*\cir<2pt>{}="E4",
(19.5,8)*\cir<2pt>{}="E5",
(0,-3)*+{\zeta},
(6.5,2.5)*+{\zeta^{-1}},
(13,-3)*+{-1},
(19.5,-3)*+{\zeta},
(19.5,10)*+{-1},
(26,-3)*+{-1},
(39,-3)*+{\zeta},
(32.5,2.5)*+{\zeta^{-1}},
(24,5.5)*+{\zeta},
(15,5.5)*+{\zeta},
\ar @{-}"E1";"E2"
\ar @{-}"E2";"E3"
\ar @{-}"E3";"E4"
\ar @{-}"E2";"E5"
\ar @{-}"E3";"E5"
\endxy$ or 
$\xy 
(0,0)*\cir<2pt>{}="E1", 
(13,0)*\cir<2pt>{}="E2",
(26,0)*\cir<2pt>{}="E3",
(39,0)*\cir<2pt>{}="E4",
(52,0)*\cir<2pt>{}="E5",
(0,-3)*+{\zeta},
(6.5,2.5)*+{\zeta^{-1}},
(13,-3)*+{\zeta},
(19.5,2.5)*+{\zeta^{-1}},
(26,-3)*+{-1},
(39,-3)*+{\zeta},
(32.5,2.5)*+{\zeta^{-1}},
(45.5,2.5)*+{\zeta^{-1}},
(52,-3)*+{\zeta},
\ar @{-}"E1";"E2"
\ar @{-}"E2";"E3"
\ar @{-}"E3";"E4"
\ar @{-}"E4";"E5"
\endxy$, type $\mathfrak{g}(2,6)$, $\zeta\in\Bbb G_3^\prime$,   see \cite[Page 477]{Andruskiewitsch2017} 
and the line 11 of Table 4  in \cite{heckenberger2009classification}.
\end{enumerate}
\end{proposition}

\begin{proof}
Let $i\in [5, n]$, then $\tilde q_{i1}=\tilde q_{i2}$ and $\tilde q_{i3}=\tilde q_{i4}$. 
Since the generalized Dynkin diagram of $(W_{X, \tilde r}, \tilde c)$ is connected, 
there exist some $i\in[5, n]$ and $j\in\{1, 3\}$
such that $\tilde q_{ij}=\tilde q_{i\tau(j)}\neq 1$. Without loss of generality, we let 
$\tilde q_{53}=\tilde q_{54}\neq 1$.  
The generalized Dynkin diagram of $(W_{X,\tilde r}, \tilde c)$ doesn't contain   subgraphs 
$\xy 
(0,0)*\cir<2pt>{}="E1", 
(6,0)*\cir<2pt>{}="E2",
(3,3)*\cir<2pt>{}="E5",
(0,6)*\cir<2pt>{}="E3",
(6,6)*\cir<2pt>{}="E4",
(-2,0)*+{1},
(8,0)*+{4},
(-2,6)*+{3},
(8,6)*+{2},
(5.5,3)*+{5},
\ar @{-}"E1";"E5"
\ar @{-}"E2";"E5"
\ar @{-}"E3";"E5"
\ar @{-}"E4";"E5"
\endxy$
and 
$\xy 
(0,0)*\cir<2pt>{}="E1", 
(6,0)*\cir<2pt>{}="E2",
(0,6)*\cir<2pt>{}="E3",
(6,6)*\cir<2pt>{}="E4",
(-2,0)*+{5},
(8,0)*+{4},
(-2,6)*+{3},
(8,6)*+{j},
\ar @{-}"E1";"E2"
\ar @{-}"E1";"E3"
\ar @{-}"E2";"E4"
\ar @{-}"E3";"E4"
\endxy$ for $j\in[6, n]$
which implies that  
\[
\tilde q_{51}=\tilde q_{52}=1,\quad
\tilde q_{j3}=\tilde q_{j4}= 1, \quad \forall j\in[6, n].
\] 
If it does not  exist some $i\in \{1, 2\}$ and $j\in \{3,4\}$ such that $\tilde q_{ij}\neq 1$, 
then there is some $k\in [6,n]$ such that $\tilde q_{k1}=\tilde q_{k2}\neq 1$ since the generalized Dynkin digram 
of $(W_{X, \tilde r}, \tilde c)$ is connected. This implies that  
$\tilde q_{k3}=\tilde q_{k4}=1$ and $\tilde q_{j1}=\tilde q_{j2}=1$ for any $j\in[5,n]-\{k\}$.
Now the generalized Dynkin diagram of $(W_{X, \tilde r}, \tilde c)$ contains a subgraph looks like 
$$
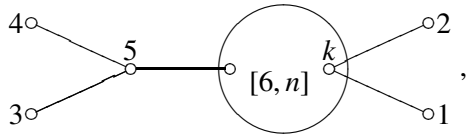
\xy 
(0,6)*\cir<2pt>{}="E3", 
(0,-6)*\cir<2pt>{}="E4", 
(13,0)*\cir<2pt>{}="E5",
(39,0)*\cir<2pt>{}="E6",
(52,6)*\cir<2pt>{}="E1",
(52,-6)*\cir<2pt>{}="E2",
(26,0)*\cir<2pt>{}="E7",
(33,0)*\cir<26pt>{}="E8",
(-2,-6)*+{3},
(-2,6)*+{4},
(13,2.5)*+{5},
(32.5,-2)*+{[6,n]},
(39,2.5)*+{k},
(54,-6)*+{1},
(54,6)*+{2},
\ar @{-}"E3";"E5"
\ar @{-}"E4";"E5"
\ar @{-}"E1";"E6"
\ar @{-}"E2";"E6"
\ar @{-}"E4";"E5"
\ar @{-}"E5";"E7"
\endxy,$$
where the biggest circle is not a part of the generalized Dynkin diagram and any vertex in $[6, n]$ 
are  placed in the biggest circle. According to Heckenberger's work, the Nichols algebra is 
infinite dimensional in this situation. 

So there exist some  $i\in \{1, 2\}$ and $j\in \{3,4\}$ such that $\tilde q_{ij}\neq 1$. 
Without loss of generality, we set $\tilde q_{13}=\tilde q_{24}\neq 1$. Now we can see that 
the generalized Dynkin diagram contains a subgraph  looks like 
$$
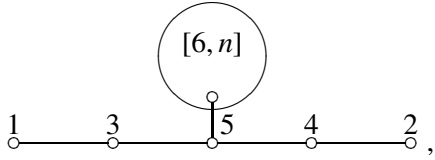
\xy 
(0,0)*\cir<2pt>{}="E1", 
(13,0)*\cir<2pt>{}="E3",
(26,0)*\cir<2pt>{}="E5",
(39,0)*\cir<2pt>{}="E4",
(52,0)*\cir<2pt>{}="E2",
(26,6)*\cir<2pt>{}="E7",
(26,11.5)*\cir<20pt>{}="E8",
(0,2.5)*+{1},
(13,2.5)*+{3},
(28,2.5)*+{5},
(39,2.5)*+{4},
(52,2.5)*+{2},
(26,13)*+{[6,n]},
\ar @{-}"E3";"E1"
\ar @{-}"E3";"E5"
\ar @{-}"E4";"E5"
\ar @{-}"E4";"E2"
\ar @{-}"E5";"E7"
\endxy,$$
with labels satisfying 
\[
q_{11}=q_{22}, \quad q_{33}=q_{44}, \quad 
\tilde q_{13}=\tilde q_{42}, \quad \tilde q_{35}=\tilde q_{54}. 
\]
According to Heckenberger's work \cite{heckenberger2009classification}, if $n\geq 6$, then $n=6$ and the generalized Dynkin diagram is Cartan type $E_6$; 
if $n=5$, then the generalized Dynkin diagram looks like \\
$
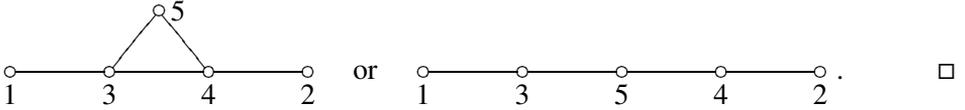
\xy 
(0,0)*\cir<2pt>{}="E1", 
(13,0)*\cir<2pt>{}="E2",
(26,0)*\cir<2pt>{}="E3",
(39,0)*\cir<2pt>{}="E4",
(19.5,8)*\cir<2pt>{}="E5",
(0,-3)*+{1},
(13,-3)*+{3},
(22,8)*+{5},
(26,-3)*+{4},
(39,-3)*+{2},
\ar @{-}"E1";"E2"
\ar @{-}"E2";"E3"
\ar @{-}"E3";"E4"
\ar @{-}"E2";"E5"
\ar @{-}"E3";"E5"
\endxy\quad\text{or}\quad 
\xy 
(0,0)*\cir<2pt>{}="E1", 
(13,0)*\cir<2pt>{}="E2",
(26,0)*\cir<2pt>{}="E3",
(39,0)*\cir<2pt>{}="E4",
(52,0)*\cir<2pt>{}="E5",
(0,-3)*+{1},
(13,-3)*+{3},
(26,-3)*+{5},
(39,-3)*+{4},
(52,-3)*+{2},
\ar @{-}"E1";"E2"
\ar @{-}"E2";"E3"
\ar @{-}"E3";"E4"
\ar @{-}"E4";"E5"
\endxy.$
\end{proof}

\begin{proposition}
Let $(X, r)$ with $r(i, j)=(\tau(j), \tau(i))$ be an involutive near-rack solution, 
where  $\tau=(1, 2n)(2,2n-1)\cdots(n,n+1)\in \Bbb S_{2n}$ with $n\geq 2$. 
Braided vector spaces $(W_{X, r}, c)$ and $(W_{X, \tilde r}, \tilde c)$ are t-equivalent as described in 
Corollary \ref{Involutive_T_equivalent}. 
If $\mathfrak{B}(W_{X, \tilde r}, \tilde c)$ 
is finite-dimensional and its generalized Dynkin diagram is connected, then $(W_{X, \tilde r}, \tilde c)$ is of 
Cartan type $A_{2n}$ or symmetric super type ${\bf A}_{2n}(q;\Bbb J)$. 
\end{proposition}
\begin{proof}
If there is a vertex $k$ connected with vertexes $i$ and $\tau(i)$, then 
$\tilde q_{ki}=q_{\tau(k)\tau(i)}\neq 1$ and $\tilde q_{\tau(k)i}=q_{k\tau(i)}\neq 1$. 
This means that the generalized Dynkin diagram of $(W_{X, \tilde r}, \tilde c)$ contains the subgraph 
$\xy 
(0,0)*\cir<2pt>{}="E1", 
(6,0)*\cir<2pt>{}="E2",
(0,6)*\cir<2pt>{}="E3",
(6,6)*\cir<2pt>{}="E4",
(-2,0)*+{k},
(8,0)*+{i},
(-4,6)*+{\tau(i)},
(10,6)*+{\tau(k)},
\ar @{-}"E1";"E2"
\ar @{-}"E1";"E3"
\ar @{-}"E2";"E4"
\ar @{-}"E3";"E4"
\endxy$, 
which is contradicted with the condition that the Nichols algebra is finite dimensional. 

Suppose that there are three vertexes $\{i_1, i_2, i_3\}\subseteq [1, 2n]$ connected with $k\in[1, 2n]$. 
Since  $\{i_1, i_2, i_3\}\bigcap \{\tau(i_1), \tau(i_2), \tau(i_3)\}$ is an empty set,  the generalized Dynkin diagram 
of $(W_{X,\tilde r}, \tilde c)$ contains  the following subgraph  
$$
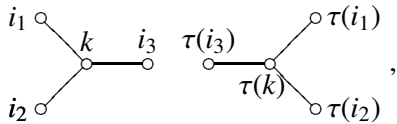
\xy 
(0,6)*\cir<2pt>{}="E1", 
(0,-6)*\cir<2pt>{}="E2",
(6,0)*\cir<2pt>{}="E3",
(14,0)*\cir<2pt>{}="E4",
(22,0)*\cir<2pt>{}="E5",
(30,0)*\cir<2pt>{}="E6",
(36,6)*\cir<2pt>{}="E7",
(36,-6)*\cir<2pt>{}="E8",
(-3,6)*+{i_1},
(-3,-6)*+{i_2},
(6,3)*+{k},
(-3,-6)*+{i_2},
(14,3)*+{i_3},
(22,3)*+{\tau(i_3)},
(29,-3)*+{\tau(k)},
(41,6)*+{\tau(i_1)},
(41,-6)*+{\tau(i_2)},
\ar @{-}"E1";"E3"
\ar @{-}"E2";"E3"
\ar @{-}"E4";"E3"
\ar @{-}"E6";"E5"
\ar @{-}"E6";"E7"
\ar @{-}"E6";"E8"
\endxy, $$
which implies that the Nichols algebra is infinite dimensional. 

So the generalized Dynkin diagram of $(W_{X, \tilde r}, \tilde c)$ looks like 
$\xy
(0,0)*\cir<2pt>{}="E1", 
(6,0)*\cir<2pt>{}="E2",
(12,0)*\cir<2pt>{}="E3",
(18,0)*\cir<2pt>{}="E4",
(9,0)*+{\cdots},
\ar @{-}"E1";"E2"
\ar @{-}"E3";"E4"
\endxy$. 
Observing that if a vertex $i$ is connected to $j$ then $\tau(i)$ is connected to $\tau(j)$, 
the rest is obvious according to Heckenberger's work \cite{heckenberger2009classification}.  
\end{proof}

\begin{proposition}\label{Length_k}
Let $(X, r)$ with $r(i, j)=(\tau(j), \tau(i))$ be an involutive near-rack solution, 
where  $\tau=(1, 2)(3,4)\cdots(2k-1,2k)\in \Bbb S_{n}$ with $n>2k\geq 6$. 
Braided vector spaces $(W_{X, r}, c)$ and $(W_{X,\tilde r}, \tilde c)$ are t-equivalent as described in 
Corollary \ref{Involutive_T_equivalent}. 
If $\mathfrak{B}(W_{X,\tilde r}, \tilde c)$ 
is finite-dimensional and its generalized Dynkin diagram is connected, then $n=2k+1$ and $(W_{X,\tilde r}, \tilde c)$ is of 
Cartan type $A_{2k+1}$ or symmetric super type ${\bf A}_{2k+1}(q;\Bbb J)$. 
\end{proposition}
\begin{proof}
Similar to the proof of Proposition \ref{LengthTwo}, the generalized Dynkin diagram of $(W_{X,\tilde r}, \tilde c)$ looks like 
$$
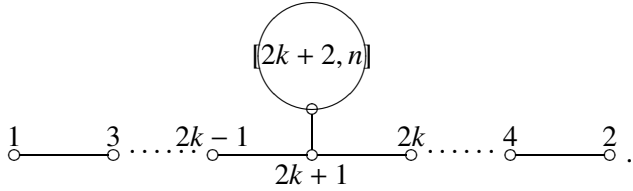
\xy 
(-13,0)*\cir<2pt>{}="E0", 
(0,0)*\cir<2pt>{}="E1", 
(13,0)*\cir<2pt>{}="E3",
(26,0)*\cir<2pt>{}="E5",
(39,0)*\cir<2pt>{}="E4",
(52,0)*\cir<2pt>{}="E2",
(65,0)*\cir<2pt>{}="E00",
(26,6)*\cir<2pt>{}="E7",
(26,13)*\cir<23pt>{}="E8",
(-13,2.5)*+{1},
(0,2.5)*+{3},
(6.5,0)*+{\cdots\cdots},
(13,2.5)*+{2k-1},
(26,-2.5)*+{2k+1},
(39,2.5)*+{2k},
(45.5,0)*+{\cdots\cdots},
(52,2.5)*+{4},
(65,2.5)*+{2},
(26,13)*+{[2k+2,n]},
\ar @{-}"E0";"E1"
\ar @{-}"E3";"E5"
\ar @{-}"E4";"E5"
\ar @{-}"E5";"E7"
\ar @{-}"E2";"E00"
\endxy.$$
According to \cite{heckenberger2009classification}, $n=2k+1$ and the generalized Dynkin diagram 
of $(W_{X,\tilde r}, \tilde c)$
is of Cartan type $A_{2k+1}$ or symmetric super type ${\bf A}_{2k+1}(q;\Bbb J)$. 
\end{proof}

\section{Appendix}
We present examples of t-equivalence here.  First we give 
necessary and sufficient conditions for $W_{X, r}$ to be a braided vector space
according to Formula \eqref{YBEquation}. Then we obtain a t-equivalence between 
$(W_{X, r}, c)$ and  some $(W_{X, \tilde r},\tilde c)$ according to Formula \eqref{T-EquEq}. 
The calculations are carried out by the software SageMath. 
\begin{example}
\label{T_equiv_DihedralD3}
Let $(X, r)$  be a near-rack solution with $r(i, j)=(\sigma_i(j), \tau(i))$, where 
$\sigma_1={\rm id}$, $\sigma_2=(1,3,2)$,  $\sigma_3=(1,2,3)$, 
$\tau=(2,3)\in\Bbb S_3$. 
Then $W_{X, r}=\bigoplus_{p=1}^3\k w_p$ is a braided vector space with a braiding given by 
\[
c(w_i\otimes w_j)=x_{3(i-1)+j} w_{\sigma_i(j)}\otimes w_{\tau(i)}, \quad \forall i, j\in[1,3],
\]
where $x_k\in\Bbbk^\times$ for $k\in[1,9]$ and
\[
x_6=x_8=x_1,\quad 
x_3=x_2^2x_1^{-1},\quad
x_7=x_1x_2x_5^{-1},\quad
x_9=x_2^2x_4^{-1},\quad x_1^3=x_2^3.
\]
If $x_2=x_1$, then there exists an invertible map 
$\varphi: W_{X, r}\to W_{X, r}$ via $\varphi(w_1)=w_1$, 
$\varphi(w_2)=\left(\frac{x_4x_5}{x_1^2}\right)^{\frac13}w_3$, $\varphi(w_3)=\left(\frac{x_1^2}{x_4x_5}\right)^{\frac13} w_2$, 
such that  $\tilde c=(\varphi^{-1}\otimes {\rm id})c(\varphi\otimes {\rm id})
=({\rm id}\otimes \varphi^{-1})c({\rm id}\otimes \varphi)$. So $(W_{X, r}, c)$ is t-equivalent to 
$(W_{X, \tilde r}, \tilde c)$. The braided vector space  $(W_{X, \tilde r}, \tilde c)$ is of type dihedral rack $\Bbb D_3$.  
\end{example}

\begin{example}[The unique near-rack solution related with $(1,2,3)^{{\rm Alt}_4}$]
\label{123Alt4}
Let $(X, r)$ be a near-rack given by $\sigma_1=(1,2,4)$, $\sigma_2=(1,3,2)$, $\sigma_3=(2,3,4)$, $\sigma_4=(1,4,3)$, $\tau=(1,4)(2,3)\in\Bbb S_4$. Then $W_{X, r}=\bigoplus_{p=1}^4\k w_p$ is a braided vector space with a braiding given by 
\[
c(w_i\otimes w_j)=x_{4(i-1)+j} w_{\sigma_i(j)}\otimes w_{\tau(i)}, \quad \forall i, j\in[1,4],
\]
where $x_k\in\Bbbk^\times$ for $k\in[1,16]$, $x_{3}^{6} x_{4}^{2} x_{5}^{2}=x_{1}^{6} x_{6}^{4}$, 
$x_{13}=x_{10}=x_7=x_4$ and
\begin{align*}
x_2&= \frac{x_{3}^{3}}{x_{1} x_{4}} ,&
x_8&= \frac{x_{3}^{2} x_{4} x_{5}}{x_{1}^{2} x_{6}} ,&
x_{9}&= \frac{x_{3} x_{4}^{2}}{x_{1} x_{6}} ,&
x_{12}&= \frac{x_{3}^{2} x_{4}^{4} x_{5}}{x_{1}^{3} x_{6}^{3}} ,\\
x_{15}&= \frac{x_{3} x_{4}^{3}}{x_{1}^{2} x_{6}} ,&
x_{11}&= \frac{x_{3} x_{4}}{x_{5}} ,&
x_{14}&= \frac{x_{3}^{2} x_{4}^{3} x_{5}}{x_{1}^{3} x_{6}^{2}} ,&
x_{16}&= \frac{x_{4}^{3} x_{5}}{x_{1} x_{3} x_{6}}.
\end{align*}
Define an invertible map 
$\varphi: W_{X, r}\to W_{X, r}$ via $\varphi(w_1)=w_4$, 
$\varphi(w_2)=\frac{x_{1} x_{6}}{x_{3}^{2}} w_3$, 
$\varphi(w_3)=\frac{x_{3}^{3} x_{4}^{3} x_{5}}{x_{1}^{4} x_{6}^{3}} w_2$, 
$\varphi(w_4)=\frac{x_{3} x_{4}^{3} x_{5}}{x_{1}^{3} x_{6}^{2}} w_1$, 
then  $\tilde c=(\varphi^{-1}\otimes {\rm id})c(\varphi\otimes {\rm id})
=({\rm id}\otimes \varphi^{-1})c({\rm id}\otimes \varphi)$. So $(W_{X, r}, c)$ is t-equivalent to 
$(W_{X, \tilde r}, \tilde c)$. 
\end{example}

\begin{example}[First near-rack solution related with $(1, 2, 3, 4)^{\Bbb S_4}$]
\label{1234S4}
Let $(X, r)$ be a near-rack given by $\sigma_1=(2,3,5,4)$, $\sigma_2=(1,3)(2,5)(4,6)$, $\sigma_3=(1,5)(2,6)(3,4)$, 
$\sigma_4=(1,2)(3,4)(5,6)$, $\sigma_5=(1,4)(2,5)(3,6)$, $\sigma_6=(2,4,5,3)$, 
$\tau=(2,5)(3,4) \in \Bbb S_6$. Then $W_{X, r}=\bigoplus_{p=1}^6\k w_p$ is a braided vector space with a braiding given by 
\[
c(w_i\otimes w_j)=x_{6(i-1)+j} w_{\sigma_i(j)}\otimes w_{\tau(i)}, \quad \forall i, j\in[1,6],
\]
where $x_k\in\Bbbk^\times$ for $k\in[1,36]$, $x_1^4 =x_6^4$, $x_{36}=x_{26}=x_{21}=x_{16}=x_{11}=x_1$ and
\begin{align*}
x_5&= \frac{x_{1}^{4}}{x_{2} x_{3} x_{4}} ,&
x_8&= \frac{x_{2}^{2} x_{3}^{2}}{x_{1}^{2} x_{6}} ,&
x_{29}&= \frac{x_{1}^{6}}{x_{2}^{2} x_{3}^{2} x_{6}} ,&
x_{9}&= \frac{x_{2}^{2} x_{3}^{2}}{x_{1} x_{6} x_{7}} ,\\
x_{32}&= \frac{x_{1} x_{6}}{x_{4}} ,&
x_{12}&= \frac{x_{2}^{2} x_{3}^{2}}{x_{1} x_{10} x_{6}} ,&
x_{31}&= x_{6} ,&
x_{13}&= \frac{x_{1}^{3} x_{7}}{x_{2}^{2} x_{4}} ,\\
x_{17}&= \frac{x_{1}^{4}}{x_{4} x_{6} x_{7}} ,&
x_{14}&= \frac{x_{1}^{4} x_{10} x_{6}}{x_{2}^{2} x_{3} x_{4}^{2}} ,&
x_{18}&= \frac{x_{1}^{3} x_{3}}{x_{10} x_{6}^{2}} ,&
x_{25}&= \frac{x_{1}^{6} x_{7}}{x_{2}^{3} x_{3}^{2} x_{4}} ,\\
x_{27}&= \frac{x_{1}^{4} x_{10} x_{6}^{2}}{x_{2}^{3} x_{3}^{2} x_{4}} ,&
x_{28}&= \frac{x_{1} x_{2} x_{4}}{x_{6} x_{7}} ,&
x_{30}&= \frac{x_{1}^{3} x_{2} x_{4}}{x_{10} x_{6}^{3}} ,&
x_{19}&= \frac{x_{1} x_{4} x_{7}}{x_{2} x_{3}} ,\\
x_{20}&= \frac{x_{2}^{3} x_{3} x_{4}}{x_{1}^{2} x_{6} x_{7}} ,&
x_{23}&= \frac{x_{1}^{4} x_{10}}{x_{2} x_{3}^{2} x_{6}} ,&
x_{24}&= \frac{x_{2}^{3} x_{3}^{2} x_{4}^{2}}{x_{1}^{5} x_{10}} ,&
x_{15}&= \frac{x_{1}^{6}}{x_{2}^{2} x_{4}^{2} x_{6}} ,\\
x_{22}&= \frac{x_{2}^{2} x_{4}^{2}}{x_{1}^{2} x_{6}} ,&
x_{33}&= \frac{x_{1} x_{6}}{x_{2}} ,&
x_{34}&= \frac{x_{2} x_{3} x_{4} x_{6}}{x_{1}^{3}} ,&
x_{35}&= \frac{x_{1} x_{6}}{x_{3}}.
\end{align*}
Define an invertible map 
$\varphi: W_{X, r}\to W_{X, r}$ via $\varphi(w_1)=w_1$, 
$\varphi(w_2)=\frac{x_{2} x_{3}}{x_{1}^{2}} w_5$, 
$\varphi(w_3)=\frac{x_{1}^{2}}{x_{2} x_{4}} w_4$, 
$\varphi(w_4)=\frac{x_{2} x_{4}}{x_{1}^{2}} w_3$, 
$\varphi(w_5)=\frac{x_{1}^{2}}{x_{2} x_{3}} w_2$, 
$\varphi(w_6)=\frac{x_{6}^{2}}{x_{1}^{2}} w_6$, 
then  $\tilde c=(\varphi^{-1}\otimes {\rm id})c(\varphi\otimes {\rm id})
=({\rm id}\otimes \varphi^{-1})c({\rm id}\otimes \varphi)$. So $(W_{X, r}, c)$ is t-equivalent to 
$(W_{X, \tilde r}, \tilde c)$. 
\end{example}

\begin{example}[Second near-rack solution related with $(1, 2, 3, 4)^{\Bbb S_4}$]
\label{T_Equiva_Second1234S4}
Let $(X, r)$ be a near-rack solution given by
$\sigma_1=(1,2,3)(4,6,5)$, $\sigma_2=(1,6)(2,5)$, $\sigma_3=(1,3,5)(2,6,4)$, 
$\sigma_4=(1,3,2)(4,5,6)$, $\sigma_5=(2,5)(3,4)$, $\sigma_5=(1,5,3)(2,4,6)$, 
$\tau=(1,3)(2,5)(4,6)\in \Bbb S_6$. Then $W_{X, r}=\bigoplus_{p=1}^6\k w_p$ is a braided vector space with a braiding given by 
\[
c(w_i\otimes w_j)=x_{6(i-1)+j} w_{\sigma_i(j)}\otimes w_{\tau(i)}, \quad \forall i, j\in[1,6],
\]
where $x_k\in\Bbbk^\times$ for $k\in[1,36]$,  $x_{34}=x_{26}=x_{24}=x_{13}=x_{11}=x_3$ and 
\begin{align*}
x_{6}&= \frac{x_{1} x_{2} x_{3}}{x_{4} x_{5}} ,&
x_7&= \frac{q_{2} x_{4} x_{5} x_{9}^{2}}{x_{2} x_{3}^{2}} ,&
x_8&= \frac{q_{3} x_{9}^{2}}{x_{3}} ,&
x_{10}&=  x_{9} ,\\
x_{12}&= \frac{q_{3}^{2} x_{2} x_{3}^{2}}{q_{2} x_{4} x_{5}} ,&
x_{14}&= \frac{q_{2} x_{4} x_{5} x_{9}}{q_{3}^{2} x_{1} x_{2}} ,&
x_{15}&= \frac{x_{3} x_{9}}{x_{2}} ,&
x_{16}&= \frac{x_{3}^{3}}{q_{3} x_{5} x_{9}} ,\\
x_{17}&= \frac{x_{3}^{3}}{q_{2} x_{1} x_{9}} ,&
x_{18}&= \frac{x_{3}^{2}}{ x_{4}} ,&
x_{19}&= x_{4} ,&
x_{20}&= \frac{x_{3} x_{4}}{q_{3} x_{1}} ,\\
x_{21}&= \frac{q_{3} x_{3} x_{4}}{x_{2}} ,&
x_{22}&= \frac{q_{3} x_{3} x_{4}}{x_{5}} ,&
x_{23}&= \frac{x_{4}^{2} x_{5}}{q_{3} x_{1} x_{2}} ,&
x_{25}&= \frac{x_{3}^{2}}{q_{2} x_{9}} ,\\
x_{27}&= \frac{q_{3} x_{3} x_{5}}{q_{2} x_{2}} ,&
x_{28}&= \frac{ x_{2} x_{3}^{3}}{q_{2} x_{5} x_{9}^{2}} ,&
x_{29}&= \frac{q_{3}^{2} x_{3}^{3}}{ x_{9}^{2}} ,&
x_{30}&= \frac{q_{2} x_{3}^{2}}{x_{9}} ,\\
x_{31}&= \frac{q_{3}^{2} x_{1} x_{9}}{q_{2} x_{4}} ,&
x_{32}&= \frac{ x_{5} x_{9}}{q_{3} x_{4}} ,&
x_{33}&= \frac{ x_{3}^{2}}{x_{4}} ,&
x_{35}&= \frac{q_{3} x_{2} x_{3}^{2}}{ x_{4} x_{9}} ,\\
x_{36}&= \frac{q_{2} q_{3} x_{1} x_{2} x_{3}^{3}}{x_{4}^{2} x_{5} x_{9}} ,&
q_2^2&=q_3^3=1.
\end{align*}
Suppose $q_3=1$,  we define an invertible map  
$\varphi: W_{X, r}\to W_{X, r}$ via 
$\varphi(w_1)=z_1w_3$, 
$\varphi(w_2)=\frac{\sqrt[3]{q_{2}} x_{9} z_{1}}{\sqrt[3]{x_{1} x_{2} x_{3}}} w_5$, 
$\varphi(w_3)=\frac{x_{3} z_{1} \sqrt[3]{x_{1} x_{2} x_{3}}}{\sqrt[3]{q_{2}} x_{1} x_{2}} w_1$, 
$\varphi(w_4)=\frac{q_{2}^{\frac{2}{3}} x_{4} z_{1} \sqrt[3]{x_{1} x_{2} x_{3}}}{q_{3}^{2} x_{1} x_{2}} w_6$, 
$\varphi(w_5)=\frac{q_{3} x_{3}^{2} z_{1}}{q_{2}^{\frac{2}{3}} x_{9} \sqrt[3]{x_{1} x_{2} x_{3}}} w_2$, 
$\varphi(w_6)=\frac{x_{3} z_{1}}{q_{2} x_{4}} w_4$, 
then  $\tilde c=(\varphi^{-1}\otimes {\rm id})c(\varphi\otimes {\rm id})
=({\rm id}\otimes \varphi^{-1})c({\rm id}\otimes \varphi)$. So $(W_{X, r}, c)$ is t-equivalent to 
$(W_{X, \tilde r}, \tilde c)$. 
\end{example}

\begin{example}[The unique near-rack solution related with Aff(5,2)]
\label{Aff52}
Let (X, r) be a near-rack solution given by $\sigma_1=(2,4,5,3)$, $\sigma_2=(1,5,2,3)$, $\sigma_3=(1,4,3,5)$, $\sigma_4=(1,3,4,2)$,  $\sigma_5=(1,2,5,4)$, $\tau=(2,5)(3,4)\in\Bbb S_5$.  
Then $W_{X, r}=\bigoplus_{p=1}^5\k w_p$ is a braided vector space with a braiding given by 
\[
c(w_i\otimes w_j)=x_{5(i-1)+j} w_{\sigma_i(j)}\otimes w_{\tau(i)}, \quad \forall i, j\in[1,5],
\]
where $x_k\in\Bbbk^\times$ for $k\in[1,25]$, $x_{22}=x_{18}=x_{14}=x_{10}=x_1$ and 
\begin{align*}
x_{5}&= \frac{x_{1}^{4}}{x_{2} x_{3} x_{4}} ,&
x_{7}&= \frac{x_{2} x_{4}}{x_{3}} ,&
x_8&= \frac{x_{2}^{3} x_{3} x_{4}^{3}}{x_{1}^{5} x_{6}} ,&
x_9&= \frac{x_{2} x_{4}}{x_{1}} ,\\
x_{11}&= \frac{x_{1} x_{3} x_{6}}{x_{4}^{2}} ,&
x_{12}&= \frac{x_{2} x_{3}}{x_{1}} ,&
x_{13}&= \frac{x_{2}^{2} x_{3}^{2} x_{4}}{x_{1}^{4}} ,&
x_{15}&= \frac{x_{2}^{2} x_{3} x_{4}}{x_{1}^{2} x_{6}} ,\\
x_{16}&= \frac{x_{1}^{7} x_{6}}{x_{2}^{3} x_{3}^{2} x_{4}^{2}} ,&
x_{17}&= \frac{x_{2} x_{4}^{2}}{x_{3} x_{6}} ,&
x_{19}&= \frac{x_{1}^{4}}{x_{2}^{2} x_{3}} ,&
x_{20}&= \frac{x_{1}^{3}}{x_{2} x_{3}} ,\\
x_{21}&= \frac{x_{1}^{6} x_{6}}{x_{2}^{3} x_{4}^{3}} ,&
x_{23}&= \frac{x_{1}^{3}}{x_{2} x_{4}} ,&
x_{24}&= \frac{x_{1} x_{4}}{x_{6}} ,&
x_{25}&= \frac{x_{1}^{4}}{x_{2} x_{4}^{2}}.
\end{align*}
Define an invertible map  
$\varphi: W_{X, r}\to W_{X, r}$ via 
$\varphi(w_1)=z_1w_1$, 
$\varphi(w_2)=\frac{x_{2} x_{4} z_{1}}{x_{1}^{2}} w_5$, 
$\varphi(w_3)=\frac{x_{2} x_{3} z_{1}}{x_{1}^{2}} w_4$, 
$\varphi(w_4)=\frac{x_{1}^{2} z_{1}}{x_{2} x_{3}} w_3$, 
$\varphi(w_5)=\frac{x_{1}^{2} z_{1}}{x_{2} x_{4}} w_2$, 
then  $\tilde c=(\varphi^{-1}\otimes {\rm id})c(\varphi\otimes {\rm id})
=({\rm id}\otimes \varphi^{-1})c({\rm id}\otimes \varphi)$. So $(W_{X, r}, c)$ is t-equivalent to 
$(W_{X, \tilde r}, \tilde c)$. 
\end{example}

\begin{example}[The unique near-rack solution related with Aff(5,3)]
\label{Aff53}
Let $(X, r)$ be a near-rack given by 
$\sigma_1= (2,3,5,4)$, $\sigma_2=(1,4,5,2)$, $\sigma_3=(1,2,4,3)$, 
$\sigma_4=(1,5,3,4)$, $\sigma_5=(1,3,2,5)$, $\tau=(2,5)(3,4)\in\Bbb S_5$. 
Then $W_{X, r}=\bigoplus_{p=1}^5\k w_p$ is a braided vector space with a braiding given by 
\[
c(w_i\otimes w_j)=x_{5(i-1)+j} w_{\sigma_i(j)}\otimes w_{\tau(i)}, \quad \forall i, j\in[1,5],
\]
where $x_k\in\Bbbk^\times$ for $k\in[1,25]$, $x_{22}=x_{18}=x_{14}=x_{10}=x_1$ and
\begin{align*}
x_{5}&= \frac{x_{1}^{4}}{x_{2} x_{3} x_{4}} ,&
x_{7}&= \frac{x_{2}^{2} x_{3}^{2}}{x_{1} x_{4} x_{6}} ,&
x_8&= \frac{x_{2} x_{3}}{x_{1}} ,&
x_{9}&= \frac{x_{2}^{2} x_{3}^{2} x_{4}}{x_{1}^{4}} ,\\
x_{11}&= \frac{x_{1}^{3} x_{6}}{x_{2}^{2} x_{3}} ,&
x_{12}&= \frac{x_{1}^{4}}{x_{2} x_{4}^{2}} ,&
x_{13}&= \frac{x_{1}^{4} x_{3}}{x_{2} x_{4}^{2} x_{6}} ,&
x_{15}&= \frac{x_{1}^{3}}{x_{2} x_{4}} ,\\
x_{16}&= \frac{x_{2} x_{4}^{2} x_{6}}{x_{1}^{3}} ,&
x_{17}&= \frac{x_{2} x_{4}}{x_{1}} ,&
x_{19}&= \frac{x_{2}^{2} x_{3} x_{4}}{x_{1}^{2} x_{6}} ,&
x_{20}&= \frac{x_{2} x_{4}}{x_{3}} ,\\
x_{21}&= \frac{x_{1}^{2} x_{4} x_{6}}{x_{2} x_{3}^{2}} ,&
x_{23}&= \frac{x_{1}^{4}}{x_{2}^{2} x_{3}} ,&
x_{24}&= \frac{x_{1}^{3}}{x_{2} x_{3}} ,&
x_{25}&= \frac{x_{1}^{5}}{x_{2} x_{3} x_{4} x_{6}}.
\end{align*}
Define an invertible map  
$\varphi: W_{X, r}\to W_{X, r}$ via 
$\varphi(w_1)=z_1w_1$, 
$\varphi(w_2)=\frac{x_{2} x_{3} z_{1}}{x_{1}^{2}} w_5$, 
$\varphi(w_3)=\frac{x_{1}^{2} z_{1}}{x_{2} x_{4}} w_4$, 
$\varphi(w_4)=\frac{x_{2} x_{4} z_{1}}{x_{1}^{2}} w_3$, 
$\varphi(w_5)=\frac{x_{1}^{2} z_{1}}{x_{2} x_{3}} w_2$, 
then  $\tilde c=(\varphi^{-1}\otimes {\rm id})c(\varphi\otimes {\rm id})
=({\rm id}\otimes \varphi^{-1})c({\rm id}\otimes \varphi)$. So $(W_{X, r}, c)$ is t-equivalent to 
$(W_{X, \tilde r}, \tilde c)$. 
\end{example}

\begin{example}[The first near-rack solution related with $(1,2)^{\Bbb S_4}$]
\label{T_Equiva_First12S4}
Let $(X, r)$ be the near-rack solution given by 
$\sigma_1={\rm id}$, $\sigma_2=(1,4,2)(3,5,6)$, $\sigma_3=(1,5,3)(2,4,6)$, 
$\sigma_4=(1,2,4)(3,6,5)$, $\sigma_5=(1,3,5)(2,6,4)$, $\sigma_6=(2,5)(3,4)$, 
$\tau=(2,4)(3,5)$. Then $W_{X, r}=\bigoplus_{p=1}^6\k w_p$ is a braided vector space with a braiding given by 
\[
c(w_i\otimes w_j)=x_{6(i-1)+j} w_{\sigma_i(j)}\otimes w_{\tau(i)}, \quad \forall i, j\in[1,6],
\]
where $x_k\in\Bbbk^\times$ for $k\in[1,36]$, $x_{36}=x_{17}=x_{27}=x_{10}=x_{20}=x_6=x_1$, 
$q^4=1$ and 
\begin{align*}
x_{3}&= \frac{x_{2}^{2}}{x_{1}} ,&
x_{4}&= \frac{x_{2}^{2}}{x_{1}} ,&
x_{5}&= x_{2} ,&
x_{12}&= \frac{x_{2}^{6} x_{7} x_{8}}{x_{1}^{5} x_{11} x_{9}} ,\\
x_{14}&= q^{2} x_{9} ,&
x_{15}&= \frac{q x_{1} x_{9}^{3}}{x_{13} x_{7} x_{8}} ,&
x_{16}&= \frac{x_{1}^{5} x_{11} x_{13}}{q x_{2}^{5} x_{7}} ,&
x_{18}&= \frac{x_{2}^{2} x_{9}^{2}}{x_{11} x_{13} x_{8}} ,\\
x_{19}&= \frac{x_{1} x_{2}}{x_{8}} ,&
x_{21}&= \frac{x_{1}^{2} x_{11} x_{9}}{x_{2} x_{7} x_{8}} ,&
x_{22}&= \frac{x_{2}^{2}}{x_{7}} ,&
x_{23}&= \frac{x_{2}^{3}}{x_{1} x_{9}} ,\\
x_{24}&= \frac{x_{2}^{4}}{x_{1}^{2} x_{11}} ,&
x_{25}&= \frac{x_{13} x_{2}^{2} x_{7} x_{8}}{q x_{1} x_{9}^{3}} ,&
x_{26}&= \frac{x_{1}^{4} x_{11} x_{13} x_{8}}{x_{2}^{4} x_{9}^{2}} ,&
x_{28}&= \frac{q^{2} x_{2}^{3}}{x_{1} x_{9}} ,\\
x_{29}&= \frac{x_{1} x_{2}}{x_{13}} ,&
x_{30}&= \frac{q x_{2}^{4} x_{7}}{x_{1}^{2} x_{11} x_{13}} ,&
x_{31}&= q^{2} x_{1} ,&
x_{32}&= \frac{x_{1}^{3} x_{13} x_{8}}{q x_{2}^{3} x_{9}} ,\\
x_{33}&= \frac{x_{2}^{3} x_{9}^{2}}{x_{1}^{2} x_{13} x_{8}} ,&
x_{34}&= \frac{x_{13} x_{2} x_{8}}{x_{9}^{2}} ,&
x_{35}&= \frac{q x_{1}^{3} x_{9}}{x_{13} x_{2} x_{8}},&
x_1^3&=x_2^3.
\end{align*}
If $x_1=x_2$, $q^2=1$, we define an invertible map 
$\varphi: W_{X, r}\to W_{X, r}$ via 
$\varphi(w_1)=z_1w_1$, 
$\varphi(w_2)=z_{1} \sqrt[3]{\frac{x_{7} x_{8}}{x_{1}^{2}}} w_4$, 
$\varphi(w_3)=\frac{x_{9} z_{1}}{q x_{1} \sqrt[3]{\frac{x_{7} x_{8}}{x_{1}^{2}}}} w_5$, 
$\varphi(w_4)=\frac{z_{1}}{\sqrt[3]{\frac{x_{7} x_{8}}{x_{1}^{2}}}} w_2$, 
$\varphi(w_5)=\frac{q x_{1} z_{1} \sqrt[3]{\frac{x_{7} x_{8}}{x_{1}^{2}}}}{x_{9}} w_3$, 
$\varphi(w_6)=\frac{z_{1}}{q} w_6$, 
then  $\tilde c=(\varphi^{-1}\otimes {\rm id})c(\varphi\otimes {\rm id})
=({\rm id}\otimes \varphi^{-1})c({\rm id}\otimes \varphi)$. So $(W_{X, r}, c)$ is t-equivalent to 
$(W_{X, \tilde r}, \tilde c)$. 
\end{example}

\begin{example}[The second near-rack solution related with $(1,2)^{\Bbb S_4}$]
\label{12S4}
Let $(X, r)$ be the near-rack solution given by 
$\sigma_1=(2,3)(4,5)$, $\sigma_2=(1,4,6,3)(2,5)$, $\sigma_3=(1,5,6,2)(3,4)$, 
$\sigma_4=(1,2,6,5)(3,4)$,  $\sigma_5=(1,3,6,4)(2,5)$, $\sigma_6=(2,4)(3,5)$, 
$\tau=(2,5)(3,4)$. 
Then $W_{X, r}=\bigoplus_{p=1}^6\k w_p$ is a braided vector space with a braiding given by 
\[
c(w_i\otimes w_j)=x_{6(i-1)+j} w_{\sigma_i(j)}\otimes w_{\tau(i)}, \quad \forall i, j\in[1,6],
\]
where $x_k\in\Bbbk^\times$ for $k\in[1,36]$, $x_{36}=x_{27}=x_{20}=x_{17}=x_{10}=x_6=x_1$, $q^4=1$ and 
\begin{align*}
x_{3}&= \frac{x_{2}^{2}}{x_{1}} ,&
x_{4}&= \frac{x_{2}^{2}}{x_{1}} ,&
x_{5}&= x_{2} ,&
x_{12}&= \frac{x_{1} x_{7} x_{8}}{x_{11} x_{9}} ,\\
x_{14}&= q^{2} x_{9} ,&
x_{15}&= \frac{q x_{1} x_{9}^{3}}{x_{13} x_{7} x_{8}} ,&
x_{16}&= \frac{x_{11} x_{13} x_{2}}{q x_{1} x_{7}} ,&
x_{18}&= \frac{x_{1}^{3} x_{9}^{2}}{x_{11} x_{13} x_{2} x_{8}} ,\\
x_{19}&= \frac{x_{1} x_{2}}{x_{8}} ,&
x_{21}&= \frac{x_{11} x_{2}^{2} x_{9}}{x_{1} x_{7} x_{8}} ,&
x_{22}&= \frac{x_{2}^{2}}{x_{7}} ,&
x_{23}&= \frac{x_{1}^{2}}{x_{9}} ,\\
x_{24}&= \frac{x_{1} x_{2}}{x_{11}} ,&
x_{25}&= \frac{x_{13} x_{2}^{2} x_{7} x_{8}}{q x_{1} x_{9}^{3}} ,&
x_{26}&= \frac{x_{1} x_{11} x_{13} x_{8}}{x_{2} x_{9}^{2}} ,&
x_{28}&= \frac{q^{2} x_{1}^{2}}{x_{9}} ,\\
x_{29}&= \frac{x_{1} x_{2}}{x_{13}} ,&
x_{30}&= \frac{q x_{1}^{4} x_{7}}{x_{11} x_{13} x_{2}^{2}} ,&
x_{31}&= q^{2} x_{1} ,&
x_{32}&= \frac{x_{13} x_{8}}{q x_{9}} ,\\
x_{33}&= \frac{x_{1} x_{9}^{2}}{x_{13} x_{8}} ,&
x_{34}&= \frac{x_{13} x_{2} x_{8}}{x_{9}^{2}} ,&
x_{35}&= \frac{q x_{1}^{3} x_{9}}{x_{13} x_{2} x_{8}} ,&
x_2^3&=x_1^3. 
\end{align*}
Define an invertible map 
$\varphi: W_{X, r}\to W_{X, r}$ via 
$\varphi(w_1)=z_1w_1$, 
$\varphi(w_2)=\frac{q x_{13} x_{8} z_{1}}{x_{2} x_{9}} w_5$, 
$\varphi(w_3)=\frac{x_{2} x_{9}^{2} z_{1}}{q^{2} x_{1} x_{13} x_{8}} w_4$, 
$\varphi(w_4)=\frac{q^{2} x_{1} x_{13} x_{8} z_{1}}{x_{2} x_{9}^{2}} w_3$, 
$\varphi(w_5)=\frac{x_{2} x_{9} z_{1}}{q x_{13} x_{8}} w_2$, 
$\varphi(w_6)=q^{2} z_{1} w_6$, 
then  $\tilde c=(\varphi^{-1}\otimes {\rm id})c(\varphi\otimes {\rm id})
=({\rm id}\otimes \varphi^{-1})c({\rm id}\otimes \varphi)$. So $(W_{X, r}, c)$ is t-equivalent to 
$(W_{X, \tilde r}, \tilde c)$. 
\end{example}

\begin{example}[The unique near-rack solution related with Aff(7,3)]
\label{Aff73}
Let $(X, r)$ be a near-rack solution given by 
$\sigma_1=(2,5,3)(4,6,7)$,  $\sigma_2=(1,6,5)(2,3,7)$,  $\sigma_3=(1,4,2)(3,5,6)$, 
  $\sigma_4=(1,2,6)(4,7,5)$,  $\sigma_5=(1,7,3)(2,4,5)$,  $\sigma_6=(1,5,7)(3,6,4)$, 
   $\sigma_7=(1,3,4)(2,7,6)$, $\tau=(2,7)(3,6)(4,5)\in\Bbb S_7$.
Then $W_{X, r}=\bigoplus_{p=1}^7\k w_p$ is a braided vector space with a braiding given by 
\[
c(w_i\otimes w_j)=x_{7(i-1)+j} w_{\sigma_i(j)}\otimes w_{\tau(i)}, \quad \forall i, j\in[1,7],
\]
where $x_k\in\Bbbk^\times$ for $k\in[1,49]$,   
$x_{44}=x_{38}=x_{32}=x_{26}=x_{20}=x_{14}=x_1$ and 
\begin{align*}
x_{5}&= \frac{x_{1}^{3}}{x_{2} x_{3}} ,&
x_{7}&= \frac{x_{1}^{3}}{x_{4} x_{6}} ,&
x_{10}&= \frac{x_{3}^{3} x_{9}^{2}}{x_{1}^{4}} ,&
x_{11}&= \frac{x_{3} x_{9}}{x_{1}} ,\\
x_{12}&= \frac{x_{3}^{3} x_{9}^{3}}{x_{1}^{2} x_{2} x_{6} x_{8}} ,&
x_{13}&= \frac{x_{2} x_{6}}{x_{1}} ,&
x_{15}&= \frac{x_{1} x_{3} x_{8}}{x_{4} x_{6}} ,&
x_{16}&= \frac{x_{3}^{4} x_{9}^{3}}{x_{1}^{3} x_{6}^{2} x_{8}} ,\\
x_{17}&= \frac{x_{2} x_{3}^{3} x_{9}}{x_{1}^{3} x_{6}} ,&
x_{18}&= \frac{x_{3} x_{4}}{x_{1}} ,&
x_{19}&= \frac{x_{3}^{3} x_{9}^{2}}{x_{1} x_{2} x_{6}^{2}} ,&
x_{21}&= \frac{x_{3}^{2} x_{9}}{x_{1} x_{6}} ,\\
x_{22}&= \frac{x_{1}^{2} x_{2} x_{4} x_{6}^{2} x_{8}}{x_{3}^{3} x_{9}^{3}} ,&
x_{23}&= \frac{x_{2} x_{4}}{x_{1}} ,&
x_{24}&= \frac{x_{2} x_{4} x_{6}}{x_{3} x_{9}} ,&
x_{25}&= \frac{x_{2} x_{4}^{2} x_{6}^{2}}{x_{1}^{2} x_{3} x_{9}} ,\\
x_{27}&= \frac{x_{2} x_{4} x_{6}}{x_{1} x_{8}} ,&
x_{28}&= \frac{x_{1} x_{2}^{2} x_{4} x_{6}}{x_{3}^{2} x_{9}^{2}} ,&
x_{29}&= \frac{x_{1}^{2} x_{8}}{x_{2} x_{4}} ,&
x_{30}&= \frac{x_{1}^{2} x_{3}^{2} x_{9}^{2}}{x_{2} x_{4}^{2} x_{6}^{2}} ,\\
x_{31}&= \frac{x_{3}^{4} x_{9}^{3}}{x_{1} x_{2} x_{4} x_{6}^{2} x_{8}} ,&
x_{33}&= \frac{x_{1}^{3} x_{3} x_{9}}{x_{2}^{2} x_{4} x_{6}} ,&
x_{34}&= \frac{x_{1}^{2} x_{3} x_{9}}{x_{2} x_{4} x_{6}} ,&
x_{35}&= \frac{x_{1}^{5}}{x_{2} x_{3} x_{4} x_{6}} ,\\
x_{36}&= \frac{x_{1}^{4} x_{2} x_{6}^{2} x_{8}}{x_{3}^{4} x_{9}^{3}} ,&
x_{37}&= \frac{x_{1}^{3} x_{6}}{x_{3}^{2} x_{9}} ,&
x_{39}&= \frac{x_{1}^{4} x_{4} x_{6}^{2}}{x_{3}^{4} x_{9}^{2}} ,&
x_{40}&= \frac{x_{1}^{2} x_{6}}{x_{2} x_{3}} ,\\
x_{41}&= \frac{x_{1}^{4} x_{6}}{x_{3}^{2} x_{4} x_{9}} ,&
x_{42}&= \frac{x_{1}^{3}}{x_{3} x_{8}} ,&
x_{43}&= \frac{x_{1}^{6} x_{6} x_{8}}{x_{3}^{4} x_{9}^{3}} ,&
x_{45}&= \frac{x_{1}^{2} x_{3}}{x_{4} x_{6}} ,\\
x_{46}&= \frac{x_{1} x_{4}}{x_{8}} ,&
x_{47}&= \frac{x_{1}^{3}}{x_{3} x_{9}} ,&
x_{48}&= \frac{x_{1}^{4} x_{6}}{x_{3}^{2} x_{9}^{2}} ,&
x_{49}&= \frac{x_{1}^{4}}{x_{3} x_{6} x_{9}}.
\end{align*}
Define an invertible map 
$\varphi: W_{X, r}\to W_{X, r}$ via 
$\varphi(w_1)=z_1w_1$, 
$\varphi(w_2)=\frac{x_{3} x_{9} z_{1}}{x_{1}^{2}}w_7$, 
$\varphi(w_3)=\frac{x_{3}^{2} x_{9} z_{1}}{x_{1}^{2} x_{6}} w_6$, 
$\varphi(w_4)=\frac{x_{2} x_{4} x_{6} z_{1}}{x_{1} x_{3} x_{9}} w_5$, 
$\varphi(w_5)=\frac{x_{1} x_{3} x_{9} z_{1}}{x_{2} x_{4} x_{6}} w_4$, 
$\varphi(w_6)=\frac{x_{1}^{2} x_{6} z_{1}}{x_{3}^{2} x_{9}} w_3$, 
$\varphi(w_7)=\frac{x_{1}^{2} z_{1}}{x_{3} x_{9}} w_2$, 
then  $\tilde c=(\varphi^{-1}\otimes {\rm id})c(\varphi\otimes {\rm id})
=({\rm id}\otimes \varphi^{-1})c({\rm id}\otimes \varphi)$. So $(W_{X, r}, c)$ is t-equivalent to 
$(W_{X, \tilde r}, \tilde c)$. 
\end{example}

\begin{example}[The unique near-rack solution related with Aff(7,5)]
\label{Aff75}
Let $(X, r)$ be a near-rack solution given by 
$\sigma_1=(2,3,5)(4,7,6)$, $\sigma_2=(1,4,3)(2,6,7)$, $\sigma_3=(1,7,5)(3,4,6)$, 
$\sigma_4=(1,3,7)(2,5,4)$, $\sigma_5=(1,6,2)(4,5,7)$, $\sigma_6=(1,2,4)(3,6,5)$, 
$\sigma_7=(1,5,6)(2,7,3)$, $\tau=(2,7)(3,6)(4,5)\in\Bbb S_7$. 
Then $W_{X, r}=\bigoplus_{p=1}^7\k w_p$ is a braided vector space with a braiding given by 
\[
c(w_i\otimes w_j)=x_{7(i-1)+j} w_{\sigma_i(j)}\otimes w_{\tau(i)}, \quad \forall i, j\in[1,7],
\]
where $x_k\in\Bbbk^\times$ for $k\in[1,49]$,  
$x_{44}=x_{38}=x_{32}=x_{26}=x_{20}=x_{14}=x_1$ and
\begin{align*}
x_{4}&= \frac{x_{1} x_{10}^{2} x_{2}^{2} x_{8}^{2}}{x_{6}^{3} x_{9}^{3}} ,&
x_{5}&= \frac{x_{1}^{3}}{x_{2} x_{3}} ,&
x_{7}&= \frac{x_{1}^{2} x_{6}^{2} x_{9}^{3}}{x_{10}^{2} x_{2}^{2} x_{8}^{2}} ,&
x_{11}&= \frac{x_{10}^{2} x_{2}^{3} x_{8}^{2}}{x_{6}^{3} x_{9}^{3}} ,\\
x_{12}&= \frac{x_{10} x_{2} x_{8}}{x_{6} x_{9}} ,&
x_{13}&= \frac{x_{10}^{3} x_{2}^{3} x_{8}^{3}}{x_{1} x_{6}^{3} x_{9}^{4}} ,&
x_{15}&= \frac{x_{1}^{2} x_{8}}{x_{2} x_{6}} ,&
x_{16}&= \frac{x_{1}^{2} x_{6} x_{9}^{2}}{x_{10} x_{2}^{2} x_{8}} ,\\
x_{17}&= \frac{x_{1}^{2} x_{6}^{3} x_{9}^{4}}{x_{10}^{2} x_{2}^{4} x_{8}^{2}} ,&
x_{18}&= \frac{x_{1}^{3} x_{9}^{2}}{x_{10} x_{2}^{2} x_{8}} ,&
x_{19}&= \frac{x_{1}^{3} x_{6}^{2} x_{9}^{3}}{x_{10} x_{2}^{3} x_{3} x_{8}^{2}} ,&
x_{21}&= \frac{x_{1} x_{3} x_{6}^{2} x_{9}^{3}}{x_{10}^{2} x_{2}^{2} x_{8}^{2}} ,\\
x_{22}&= \frac{x_{10} x_{2}^{3} x_{3} x_{8}^{2}}{x_{6}^{3} x_{9}^{3}} ,&
x_{23}&= \frac{x_{10} x_{2}^{3} x_{3}^{2} x_{8}}{x_{1}^{2} x_{6}^{2} x_{9}^{2}} ,&
x_{24}&= \frac{x_{10}^{2} x_{2}^{2} x_{3} x_{8}^{2}}{x_{6}^{3} x_{9}^{3}} ,&
x_{25}&= \frac{x_{1} x_{10}^{2} x_{2}^{3} x_{3} x_{8}^{2}}{x_{6}^{4} x_{9}^{4}} ,\\
x_{27}&= \frac{x_{10} x_{2}^{2} x_{3} x_{8}}{x_{6}^{2} x_{9}^{2}} ,&
x_{28}&= \frac{x_{2} x_{3}}{x_{8}} ,&
x_{29}&= \frac{x_{1}^{3} x_{6}^{2} x_{9}^{3}}{x_{10}^{2} x_{2}^{3} x_{3} x_{8}} ,&
x_{30}&= \frac{x_{1} x_{6}^{3} x_{9}^{3}}{x_{10} x_{2}^{2} x_{3} x_{8}^{2}} ,\\
x_{31}&= \frac{x_{1}^{2} x_{6}^{2} x_{9}^{2}}{x_{10} x_{2}^{2} x_{3} x_{8}} ,&
x_{33}&= \frac{x_{1}^{3} x_{6} x_{9}}{x_{2}^{2} x_{3}^{2}} ,&
x_{34}&= \frac{x_{1}^{2} x_{6}}{x_{2} x_{3}} ,&
x_{35}&= \frac{x_{1}^{2} x_{6}^{5} x_{9}^{5}}{x_{10}^{3} x_{2}^{4} x_{3} x_{8}^{3}} ,\\
x_{36}&= \frac{x_{1} x_{10} x_{2}^{2} x_{8}^{2}}{x_{6}^{2} x_{9}^{3}} ,&
x_{37}&= \frac{x_{2} x_{6}}{x_{1}} ,&
x_{39}&= \frac{x_{10}^{2} x_{2}^{3} x_{8}}{x_{6}^{2} x_{9}^{3}} ,&
x_{40}&= \frac{x_{1} x_{10} x_{2}^{2} x_{8}}{x_{3} x_{6} x_{9}^{2}} ,\\
x_{41}&= \frac{x_{10}^{2} x_{2}^{4} x_{3} x_{8}^{2}}{x_{1}^{2} x_{6}^{2} x_{9}^{4}} ,&
x_{42}&= \frac{x_{10} x_{2}^{2} x_{8}}{x_{6} x_{9}^{2}} ,&
x_{43}&= \frac{x_{1} x_{3}}{x_{10}} ,&
x_{45}&= \frac{x_{1}^{3} x_{6} x_{9}}{x_{10} x_{2}^{2} x_{8}} ,\\
x_{46}&= \frac{x_{1}^{2} x_{6} x_{9}}{x_{10} x_{2} x_{8}} ,&
x_{47}&= \frac{x_{1}^{4} x_{6}^{2} x_{9}^{3}}{x_{10}^{2} x_{2}^{3} x_{3} x_{8}^{2}} ,&
x_{48}&= \frac{x_{1} x_{6}}{x_{8}} ,&
x_{49}&= \frac{x_{1}^{2} x_{6}^{2} x_{9}^{2}}{x_{10}^{2} x_{2} x_{8}^{2}}.
\end{align*}
Define an invertible map 
$\varphi: W_{X, r}\to W_{X, r}$ via 
$\varphi(w_1)=z_1w_1$, 
$\varphi(w_2)=\frac{x_{10} x_{2} x_{8} z_{1}}{x_{1} x_{6} x_{9}}w_7$, 
$\varphi(w_3)=\frac{x_{1} x_{6} x_{9}^{2} z_{1}}{x_{10} x_{2}^{2} x_{8}} w_6$, 
$\varphi(w_4)=\frac{x_{10} x_{2}^{2} x_{3} x_{8} z_{1}}{x_{1} x_{6}^{2} x_{9}^{2}} w_5$, 
$\varphi(w_5)=\frac{x_{1} x_{6}^{2} x_{9}^{2} z_{1}}{x_{10} x_{2}^{2} x_{3} x_{8}}  w_4$, 
$\varphi(w_6)=\frac{x_{10} x_{2}^{2} x_{8} z_{1}}{x_{1} x_{6} x_{9}^{2}}  w_3$, 
$\varphi(w_7)=\frac{x_{1} x_{6} x_{9} z_{1}}{x_{10} x_{2} x_{8}} w_2$, 
then  $\tilde c=(\varphi^{-1}\otimes {\rm id})c(\varphi\otimes {\rm id})
=({\rm id}\otimes \varphi^{-1})c({\rm id}\otimes \varphi)$. So $(W_{X, r}, c)$ is t-equivalent to 
$(W_{X, \tilde r}, \tilde c)$. 
\end{example}


\begin{thebibliography}{10}
\expandafter\ifx\csname urlstyle\endcsname\relax
  \providecommand{\doi}[1]{doi:\discretionary{}{}{}#1}\else
  \providecommand{\doi}{doi:\discretionary{}{}{}\begingroup
  \urlstyle{rm}\Url}\fi

\bibitem{Andruskiewitsch2017}
Andruskiewitsch, N., Angiono, I. (2017).
\newblock On finite dimensional {N}ichols algebras of diagonal type.
\newblock \emph{Bull. Math. Sci.} 7(3):353--573.
\newblock \doi{10.1007/s13373-017-0113-x}.

\bibitem{Andruskiewitsch2018}
Andruskiewitsch, N., Giraldi, J. M.~J. (2018).
\newblock Nichols algebras that are quantum planes.
\newblock \emph{Linear and Multilinear Algebra} 66(5):961--991.
\newblock \doi{10.1080/03081087.2017.1331997}.

\bibitem{Andruskiewitsch2003MR1994219}
Andruskiewitsch, N., Gra{\~n}a, M. (2003).
\newblock From racks to pointed {H}opf algebras.
\newblock \emph{Adv. Math.} 178(2):177--243.
\newblock \doi{10.1016/S0001-8708(02)00071-3}.

\bibitem{andruskiewitsch2001pointed}
Andruskiewitsch, N., Schneider, H.-J. (2002).
\newblock Pointed {H}opf algebras.
\newblock In: \emph{New directions in {H}opf algebras}, vol.~43 of \emph{Math.
  Sci. Res. Inst. Publ.}, pp. 1--68. Cambridge Univ. Press, Cambridge.
\newblock \doi{10.2977/prims/1199403805}.

\bibitem{Angiono2013}
Angiono, I. (2013).
\newblock On {N}ichols algebras of diagonal type.
\newblock \emph{J. Reine Angew. Math.} 683:189--251.
\newblock \doi{10.1515/crelle-2011-0008}.

\bibitem{MR3420518}
Angiono, I. (2015).
\newblock A presentation by generators and relations of {N}ichols algebras of
  diagonal type and convex orders on root systems.
\newblock \emph{J. Eur. Math. Soc. (JEMS)} 17(10):2643--2671.
\newblock \doi{10.4171/JEMS/567}.

\bibitem{Angiono2019a}
Angiono, I., Garc{\'{\i}}a~Iglesias, A. (2019).
\newblock Liftings of {Nichols} algebras of diagonal type {II}: {All} liftings
  are cocycle deformations.
\newblock \emph{Selecta Mathematica. New Series} 25(1):95.
\newblock \doi{10.1007/s00029-019-0452-4}.
\newblock Id/No 5.

\bibitem{Bai2023}
Bai, C., Guo, L., Sheng, Y., Tang, R. (2023).
\newblock Post-groups, {(Lie-)Butcher} groups and the {Yang–Baxter} equation.
\newblock \emph{Math. Ann.} \doi{10.1007/s00208-023-02592-z}.

\bibitem{zbMATH07473025}
Bardakov, V.~G., Gubarev, V. (2022).
\newblock Rota-{Baxter} groups, skew left braces, and the {Yang}-{Baxter}
  equation.
\newblock \emph{J. Algebra} 596:328--351.
\newblock \doi{10.1016/j.jalgebra.2021.12.036}.

\bibitem{Baxter1972}
Baxter, R.~J. (1972).
\newblock Partition function of the eight-vertex lattice model.
\newblock \emph{Annals of Physics} 70:193--228.
\newblock \doi{10.1016/0003-4916(72)90335-1}.

\bibitem{Brieskorn1988}
Brieskorn, E. (1988).
\newblock Automorphic sets and braids and singularities.
\newblock Braids, {AMS}-{IMS}-{SIAM} {Jt}. {Summer} {Res}. {Conf}., {Santa}
  {Cruz}/{Calif}. 1986, {Contemp}. {Math}. 78, 45-115 (1988).

\bibitem{Caranti2023}
Caranti, A., Stefanello, L. (2023).
\newblock Skew braces from {Rota}-{Baxter} operators: a cohomological
  characterisation and some examples.
\newblock \emph{Annali di Matematica Pura ed Applicata. Serie Quarta}
  202(1):1--13.
\newblock \doi{10.1007/s10231-022-01230-w}.

\bibitem{Castelli2019}
Castelli, M., Catino, F., Miccoli, M.~M., Pinto, G. (2019).
\newblock Dynamical extensions of quasi-linear left cycle sets and the
  {Yang}-{Baxter} {Equation}.
\newblock \emph{Journal of Algebra and its Applications} 18(11):16.
\newblock \doi{10.1142/S0219498819502207}.
\newblock Id/No 1950220.

\bibitem{Ding2006}
Ding, L., Gu, P. (2006).
\newblock A characterization of metahomomorphisms on {Abelian} groups.
\newblock \emph{Northeastern Mathematical Journal} 22(4):383--386.

\bibitem{Ding2009}
Ding, L., Gu, P. (2009).
\newblock A theorem on the structure of metahomomorphisms on groups and its
  applications.
\newblock \emph{Acta Scientiarum Naturalium Universitatis Nankaiensis}
  42(3):104--107.

\bibitem{Gu1997}
Gu, P. (1997).
\newblock Another solution of {Yang}-{Baxter} equation on set and
  ``metahomomorphisms on groups''.
\newblock \emph{Chinese Science Bulletin} 42(22):1852--1855.
\newblock \doi{10.1007/BF02882773}.

\bibitem{zbMATH06713497}
Guarnieri, L., Vendramin, L. (2017).
\newblock Skew braces and the {Yang}-{Baxter} equation.
\newblock \emph{Math. Comput.} 86(307):2519--2534.
\newblock \doi{10.1090/mcom/3161}.

\bibitem{Guo2021}
Guo, L., Lang, H., Sheng, Y. (2021).
\newblock Integration and geometrization of {Rota}-{Baxter} {Lie} algebras.
\newblock \emph{Advances in Mathematics} 387:34.
\newblock \doi{10.1016/j.aim.2021.107834}.
\newblock Id/No 107834.

\bibitem{MR2207786}
Heckenberger, I. (2006).
\newblock The {W}eyl groupoid of a {N}ichols algebra of diagonal type.
\newblock \emph{Invent. Math.} 164(1):175--188.
\newblock \doi{10.1007/s00222-005-0474-8}.

\bibitem{heckenberger2009classification}
Heckenberger, I. (2009).
\newblock Classification of arithmetic root systems.
\newblock \emph{Adv. Math.} 220(1):59--124.
\newblock \doi{10.1016/j.aim.2008.08.005}.

\bibitem{Heckenberger2015}
Heckenberger, I., Lochmann, A., Vendramin, L. (2015).
\newblock Nichols algebras with many cubic relations.
\newblock \emph{Trans. Amer. Math. Soc.} 367(9):6315--6356.
\newblock \doi{10.1090/S0002-9947-2015-06231-X}.

\bibitem{Heckenberger[2020]copyright2020}
Heckenberger, I., Schneider, H.-J. (2020).
\newblock \emph{Hopf algebras and root systems}, vol. 247 of \emph{Mathematical
  Surveys and Monographs}.
\newblock American Mathematical Society, Providence, RI.

\bibitem{MR3605018}
Heckenberger, I., Vendramin, L. (2017).
\newblock A classification of {N}ichols algebras of semisimple
  {Y}etter-{D}rinfeld modules over non-abelian groups.
\newblock \emph{J. Eur. Math. Soc. (JEMS)} 19(2):299--356.
\newblock \doi{10.4171/JEMS/667}.

\bibitem{Heckenberger2017}
Heckenberger, I., Vendramin, L. (2017).
\newblock The classification of {N}ichols algebras over groups with finite root
  system of rank two.
\newblock \emph{J. Eur. Math. Soc. (JEMS)} 19(7):1977--2017.
\newblock \doi{10.4171/JEMS/711}.

\bibitem{Kapranov2020}
Kapranov, M., Schechtman, V. (2020).
\newblock Shuffle algebras and perverse sheaves.
\newblock \emph{Pure Appl. Math. Q.} 16(3):573--657.
\newblock \doi{10.4310/PAMQ.2020.v16.n3.a9}.

\bibitem{Lentner2021}
Lentner, S.~D. (2021).
\newblock Quantum groups and {Nichols} algebras acting on conformal field
  theories.
\newblock \emph{Advances in Mathematics} 378:72.
\newblock \doi{10.1016/j.aim.2020.107517}.
\newblock Id/No 107517.

\bibitem{LuJianghua2000MR1769723}
Lu, J.-H., Yan, M., Zhu, Y.-C. (2000).
\newblock On the set-theoretical {Y}ang-{B}axter equation.
\newblock \emph{Duke Math. J.} 104(1):1--18.
\newblock \doi{10.1215/S0012-7094-00-10411-5}.

\bibitem{MR1227098}
Lusztig, G. (1993).
\newblock \emph{Introduction to quantum groups}, vol. 110 of \emph{Progress in
  Mathematics}.
\newblock Birkh\"{a}user Boston, Inc., Boston, MA.

\bibitem{Meir2022}
Meir, E. (2022).
\newblock Geometric perspective on {N}ichols algebras.
\newblock \emph{J. Algebra} 601:390--422.
\newblock \doi{10.1016/j.jalgebra.2022.03.011}.

\bibitem{MR506406}
Nichols, W.~D. (1978).
\newblock Bialgebras of type one.
\newblock \emph{Comm. Algebra} 6(15):1521--1552.
\newblock \doi{10.1080/00927877808822306}.

\bibitem{MR1632802}
Rosso, M. (1998).
\newblock Quantum groups and quantum shuffles.
\newblock \emph{Invent. Math.} 133(2):399--416.
\newblock \doi{10.1007/s002220050249}.

\bibitem{Rump2007}
Rump, W. (2007).
\newblock Braces, radical rings, and the quatum {Yang}-{Baxter} equation.
\newblock \emph{Journal of Algebra} 307(1):153--170.
\newblock \doi{10.1016/j.jalgebra.2006.03.040}.

\bibitem{Schauenburg1992}
Schauenburg, P. (1992).
\newblock \emph{On coquasitriangular {H}opf algebras and the quantum
  {Y}ang-{B}axter equation}, vol.~67 of \emph{Algebra Berichte [Algebra
  Reports]}.
\newblock Verlag Reinhard Fischer, Munich.

\bibitem{MR1396857}
Schauenburg, P. (1996).
\newblock A characterization of the {B}orel-like subalgebras of quantum
  enveloping algebras.
\newblock \emph{Comm. Algebra} 24(9):2811--2823.
\newblock \doi{10.1080/00927879608825714}.

\bibitem{Shi2019}
Shi, Y.-X. (2019).
\newblock Finite-dimensional {H}opf algebras over the {K}ac-{P}aljutkin algebra
  {$H_8$}.
\newblock \emph{Rev. Un. Mat. Argentina} 60(1):265--298.
\newblock \doi{10.33044/revuma.v60n1a17}.

\bibitem{Shi2020even}
Shi, Y.-X. (2022).
\newblock Finite-dimensional {N}ichols algebras over the {S}uzuki algebras
  \uppercase\expandafter{\romannumeral1}: simple {Y}etter-{D}rinfeld modules of
  ${A}_{N\,2n}^{\mu\lambda}$.
\newblock \emph{Bull. Belg. Math. Soc. Simon Stevin} 29(2):207--233.
\newblock \doi{10.36045/j.bbms.211101}.

\bibitem{shi2023_NearRack}
Shi, Y.-X. (2023).
\newblock Finite-dimensional {Nichols} algebras over the {Suzuki} algebras
  \uppercase\expandafter{\romannumeral3}: simple {Yetter-Drinfeld} modules.
\newblock \emph{arXiv:2302.09769} .

\bibitem{Shi2023}
Shi, Y.-X. (2023).
\newblock Multinomial expansion and {Nichols} algebras associated to
  non-degenerate involutive solutions of the {Yang-Baxter} equation.
\newblock \emph{Comm. Algebra} 51(6):2532--2547.
\newblock \doi{10.1080/00927872.2023.2165660}.

\bibitem{Shi2020odd}
Shi, Y.-X. (2024).
\newblock Finite-dimensional {N}ichols algebras over the {S}uzuki algebras
  \uppercase\expandafter{\romannumeral2}: simple {Y}etter-{D}rinfeld modules of
  ${A}_{N\,2n+1}^{\mu\lambda}$.
\newblock \emph{arXiv:2103.06475, J. Algebra Appl.}
  \doi{10.1142/S0219498824501652}.

\bibitem{zbMATH01585085}
Soloviev, A. (2000).
\newblock Non-unitary set-theoretical solutions to the quantum {Yang}-{Baxter}
  equation.
\newblock \emph{Math. Res. Lett.} 7(5-6):577--596.
\newblock \doi{10.4310/MRL.2000.v7.n5.a4}.

\bibitem{woronowicz1989differential}
Woronowicz, S.~L. (1989).
\newblock Differential calculus on compact matrix pseudogroups (quantum
  groups).
\newblock \emph{Comm. Math. Phys.} 122(1):125--170.
\newblock \doi{10.1007/BF01221411}.

\bibitem{Yang1967}
Yang, C.~N. (1967).
\newblock Some exact results for the many-body problem in one dimension with
  repulsive delta-function interaction.
\newblock \emph{Physical Review Letters} 19:1312--1315.
\newblock \doi{10.1103/PhysRevLett.19.1312}.

\bibitem[43]{GAP4}
  The GAP~Group, \emph{GAP -- Groups, Algorithms, and Programming, 
  Version 4.11.1}; 2021, \url{https://www.gap-system.org}.
\bibitem[44]{SageMath}
  The Sage Developers,  \emph{{Sagemath}, the {Sage} {Mathematics} {Software} {System}, 
  Version 9.1}; 2021, \url{https://www.sagemath.org}.
\end{thebibliography}
\end{document}